\documentclass[final]{siamart250211}

\usepackage{lipsum}
\usepackage{amsfonts}
\usepackage{graphicx}
\usepackage{epstopdf}
\usepackage{algorithmic}
\ifpdf
  \DeclareGraphicsExtensions{.eps,.pdf,.png,.jpg}
\else
  \DeclareGraphicsExtensions{.eps}
\fi

\newsiamremark{remark}{Remark}
\newsiamremark{hypothesis}{Hypothesis}
\crefname{hypothesis}{Hypothesis}{Hypotheses}
\newsiamthm{claim}{Claim}
\newsiamremark{fact}{Fact}
\crefname{fact}{Fact}{Facts}

\headers{a posteriori error estimates for E-Adams2 method}{X. F. Hu and W. S. Wang}

\title{a posteriori error estimates for the two-step explicit exponential adams method for parabolic equations\thanks{Submitted to the editors DATE.
\funding{The first author was partially supported by the Natural Science Foundation of China (Grant Nos. 12271367, 12071419) and  the Natural Science Foundation of Shandong Province, China (Grant No. ZR2024MA056). The second author is the corresponding author. He was supported by the
Natural Science Foundation of China (Grant No. 12271367), Shanghai Science and Technology Planning Projects (Grant No. 20JC1414200), and Natural Science Foundation of Shanghai (Grant No. 20ZR1441200).}}}

\author{Xianfa Hu\footnotemark[2] \footnotemark[3]\thanks{School of Mathematics and Statistics, Suzhou University of Technology, Changshu 215500,  People's Republic of China
		(\email{zzxyhxf@163.com}.)} 
		\and Wansheng Wang\thanks{Department of Mathematics, Shanghai Normal University, Shanghai 200234, People's Republic of China
  	(\email{w.s.wang@163.com}.)}
}

\usepackage{amsopn}

\usepackage{enumerate,amsmath,amssymb,graphicx,subcaption}
\newtheorem{example}{Example}[section]
\newtheorem{assumption}{Assumption}[section]
\ifpdf
\hypersetup{
	pdftitle={An Example Article},
	pdfauthor={D. Doe, P. T. Frank, and J. E. Smith}
}
\fi

\begin{document}
	
	\maketitle
	
	\begin{abstract}
		In this paper, we derive optimal order  a posteriori error estimates for the variable step-size explicit two-step exponential Adams (E-Adams2) method for parabolic problems.  We begin by introducing an E-Adams2 approximation, defined  by  the piecewise linear approximate solutions, which  leads to  suboptimal error estimates. To recover  optimal order error estimates, we  introduce an appropriate reconstruction of the approximation with the second order residual for the explicit E-Adams2 method, which plays key roles in deriving optimal order a posteriori error estimates  for the proposed explicit method   for linear and semilinear parabolic equations. Various numerical experiments  are carried out to verify the correct convergence rates of the a posteriori quantities, and the high efficiency of the adaptive algorithm.
	\end{abstract}
	
	
	
	\section{Introduction}
	This paper is devoted to establishing  optimal order a posteriori error estimates for the two-step exponential Adams (E-Adams2) method for the following semilinear  parabolic equation
	\begin{equation}\label{nonlinear problem}\left\{
		\begin{array}{l}
			u'(t)+Au(t)=B(t,u(t)),\quad 0\leq t\leq T,\cr\noalign{\vskip1truemm} u(0)=u^0,
		\end{array}
		\right.\end{equation}
	{where}  $A: D(A)\rightarrow H$ is a positive definite, self-adjoint and  linear operator  on a Hilbert space $( H,(\cdot,\cdot))$ with domain $D(A)$ dense in $H$, $B(t, \cdot) : D(A)\rightarrow H, \ t\in [0,T]$, and the initial value $u^0\in H$. We consider the solution $u\in L^2(0,T;D(A))$ satisfying the equation  \cite{Hochbruck2010}.
	
	A posteriori error analysis plays an essential role in mesh adaptivity techniques, and it provides the computable error bounds of the difference between exact and numerical solutions to guide the adaptive meshes.  There have been extensive studies and  applications of a
	posteriori error estimates for numerical methods for elliptic problems \cite{Babusak1978,Bank1985,Ainsworth1993,Verfurth1994}.  Over the past decades, a posteriori error  theory of numerical methods for parabolic  problems has attracted considerable attention.  In \cite{Eriksson1991, Eriksson1995},  a posteriori error control for   discontinuous  Galerkin (DG)  methods  applied to such problems was established via  the   duality technique.   Notably, a posteriori error analysis  for DG time-stepping methods  of parabolic problems was further derived through an equilibrated flux approach; see \cite{Ern2015,Ern2017,Ern2019}.  Based on the \emph{reconstruction} technique, a posteriori error analysis has been systematically developed  for  classical (non-exponential) temporal discretization schemes,  such as the Crank--Nicolson (CN) method \cite{Akrivis2006,Lozinski2009,Wang2018,Wang2020,Wang2022b},  DG and Runge--Kutta--Radau methods \cite{Makri2006,Akrivis2020,Georgoulis2021},  continuous Galerkin   and Runge--Kutta collocation methods \cite{Akrivis2009,Akrivis2011},  the two-step backward differentiation formula method \cite{Akrivis2010,Wang2021,Wang2022a}.
	
	It is well known that exponential integrators permit larger temporal step-sizes and exhibit good stability for parabolic problems or
	stiff ordinary differential equations. Numerous efficient  exponential algorithms  have been proposed and analyzed  in the past decades. Two typical ones are  the integrating factor (IF) method \cite{Lawson1967,Krogstad2005,Mei2017,Isherwood2018,Ju2021}, and the exponential time differencing (ETD) method
	\cite{Cox2002,Berland2005,Hochbruck2005a,Hochbruck2005b,Hochbruck2010,Hochbruck2011,Li2016,Du2019,Wang2019,Li2020,Du2021,Fang2021,Huang2023,Wang2023}.
	The IF method uses a transformation of variables of the original equation.  The  ETD method starts  with the variation-of-constants formula, using interpolation polynomials to approximate the nonlinear term in the integral, and then  integrates the resulting approximation. For more details about the ETD method, we refer the reader to \cite{Hochbruck2010,Du2021}. An important feature of exponential integrators is that their implementations require computing the products of matrix exponentials and vectors, and the products are effectively evaluated by the Krylov subspace method \cite{Hochbruck1997,Higham2010,Gaudreault2018}, or  the Pad\'{e} approximation \cite{Moler2003,Berland2007,Higham2008}. Adaptive Krylov subspace methods have been  successfully  applied to  exponential integrators (see, e.g., \cite{Niesen2012,Gaudreault2018,Bergermann2024}). However, the related error estimates may only be applied to the Krylov subspace method.
	
	To the best of our knowledge,  the results in  \cite{Hu2024,Hu2025} provide the  first rigorous a  posteriori error analysis  for the (single-step) exponential midpoint method for parabolic equations, based on piecewise linear interpolation and  its quadratic reconstruction.  Additionally, an adaptive IF midpoint method for second order  evolution equations was formulated in \cite{Hu2026}, utilizing  reliable a posteriori error control. Exponential multistep methods of Adams-type were proposed and studied in \cite{Norsett1969,Calvo2006,Ostermann2006,Hochbruck2011,Li2020}.   Hochbruck and Ostermann \cite{Hochbruck2011} proved that the $s$-step exponential Adams method converges with  order $s$, and the Dahlquist stability properties  of  $s$-step (explicit) exponential Adams methods were analyzed in \cite{Coudiere2018}. This paper is concerned with the explicit E-Adams2 method, establishes optimal order  residual-based a posteriori error estimates for parabolic problems, and develops an adaptive algorithm. It should be pointed out that error bounds presented in this article depend  only on the discretization solutions  and the data of the problem, and thus are computable.
	
	The rest of this paper is organized as follows. Section \ref{sec2} is devoted to the formulation and assumptions of the E-Adams2 method.  Section \ref{sec3} introduces the E-Adams2 approximation $U$  defined by  a piecewise continuous linear interpolation at the  nodal values, which yields  suboptimal error estimates for linear problems;  to recover the optimal order, we present an E-Adams2 reconstruction $\hat{U}$ of $U$ with the second order residual, and   derive  optimal order residual-based   a  posteriori error estimates for the  method. Furthermore,  we extend the  theoretical results  to semilinear problems  in Section \ref{sec4}. An adaptive algorithm is developed based on effective error control in Section \ref{sec5}. Several  numerical examples in  Section \ref{sec6}  are carried out to further illustrate the effectiveness of the error estimators and the adaptive algorithm. Finally,  we make some concluding remarks in Section \ref{sec7}.
	
	\section{The two-step exponential Adams method for parabolic problems}\label{sec2} In this section, we present the two-step exponential Adams (E-Adams2) method and make some assumptions.
	\subsection{E-Adams2 method}  Let us consider  a   partition $0=t^{0}<t^{1}<\ldots <t^{N}=T$ of $[0,T]$,  $N\geq2$. For each $n=1,2,\ldots,N$, we denote  $k_n:=t^{n}-t^{n-1}$, and  $J_{n}:=(t^{n-1},t^{n}]$.  By the  variation-of-constants formula (see, for example, \cite{Hochbruck2005a,Hochbruck2010}), the exact solution  $u(t)$ of \eqref{nonlinear problem} on the interval $J_{n}$ satisfies the following integral equation
	\begin{equation}\label{variation of constants}
		\displaystyle u(t^{n})=e^{-k_{n}A}u(t^{n-1})+k_{n}\int_{0}^{1} e^{-(1-\tau)k_{n}A}B\big(t^{n-1}+k_{n}\tau,u(t^{n-1}+k_{n}\tau)\big)d\tau.
	\end{equation}
	The E-Adams2 method can be stated as follows: Given nodal approximations $U^{j}\approx u(t^j)$, we replace the nonlinear term $B(\tau,u(\tau))$ in \eqref{variation of constants} with the interpolation polynomial  $P_1$  through  the points $
	\big(t^{n-2},B(t^{n-2},U^{n-2})\big)$ and $ \big(t^{n-1},B(t^{n-1},U^{n-1})\big)$. This gives the scheme
	\begin{equation}\label{approximation nonlinear}
		U^{n}=e^{-k_{n}A}U^{n-1}+k_{n}\int_{0}^{1} e^{-(1-\tau)k_{n}A} P_{1}(t^{n-1}+k_{n}\tau)d\tau,  \quad n\geq2,
	\end{equation}
	where  $P_1$ is  given by
	\begin{equation}
		P_{1}(t^{n-1}+k_{n}\tau)=B^{n-1}+\frac{k_{n}\tau}{k_{n-1}}(B^{n-1}-B^{n-2}),
	\end{equation}
	with $B^{n-1}:=B(t^{n-1},U^{n-1})$. Computing the integrals then results in the E-Adams2  method
	\begin{equation}\label{nonlinear method}
		\bar{U}^{n}=e^{-k_{n}A}U^{n-1}+k_{n}\varphi_1(-k_{n}A)B^{n-1}+k_{n}r_{n}\varphi_2(-k_{n}A)\big(B^{n-1}-B^{n-2}\big),
	\end{equation}
	with $U^0:=u_0$ and $r_{n}=k_{n}/k_{n-1}$, $n=2,\ldots,N$,  where   the  $\varphi$-functions are  defined by
	\begin{equation}\label{phifunction}
		\varphi_{j}(z)=\int_{0}^{1} e^{(1-\tau)z}\frac{\tau^{j-1}}{(j-1)!}d\tau, \quad j\geq1.
	\end{equation}
	For $j\geq1$, these functions satisfy  the  recurrence relation
	\begin{equation}\label{recurrence}
		\varphi_{j+1}(z)=\frac{\varphi_j(z)-\varphi_j(0)}{z},  \quad  \varphi_j(0)=1/j!, \quad  \varphi_0(z)=e^z,
	\end{equation}
	and  have the power series expression
	\begin{equation*}
		\varphi_{j+1}(z)=\sum\limits_{l=0}^{\infty} \frac{z^l}{(l+j+1)!}.
	\end{equation*}
	
	\begin{remark} 	Unlike the single-step exponential midpoint method in \cite{Hu2024} (which involves only $\varphi_1(-k_nA)$) and the integrating factor midpoint method in \cite{Hu2026}, the present two-step scheme \eqref{nonlinear method} contains an additional $\varphi_2$-term. From \eqref{approximation nonlinear}, it is clear that the E-Adams2 method makes use of the properties of the exponential and
		the  $\varphi$-functions, which can exactly integrate  the linear part of \eqref{nonlinear problem}.  For stiff and highly oscillatory problems, the exponential contains the full information of the  linear operator $A$, and  exponential Adams methods integrate such problems explicitly. 
	\end{remark}

 \subsection{Assumption and preliminaries} Throughout the paper,  the a posteriori error analysis is  based on the following assumption for the linear operator $A$;  see \cite{Henry1981,Hochbruck2010}.
\begin{assumption}\label{assumption2.1}
Let $A:D(A)\rightarrow H$ be a sectorial operator on the Hilbert  space $H$ with the norm $\|\cdot\|_{H}$, i.e., there exist constants $M\geq1$, $a\in \mathbb{R}$, and $0 < \vartheta < \pi/2$ such that the resolvent estimate holds on the sector $\{ \lambda_{1} \in \mathbb{C} :  \vartheta \leq |\arg(\lambda_{1}-a)|\leq \pi,  \ \lambda_1 \neq a \}$,
 \begin{equation*}
  \|(\lambda_{1}I-A)^{-1}\|_{H\leftarrow H} \leq \frac{M}{|\lambda_{1}-a|},
 \end{equation*}
where the notation  $\|\cdot\|_{H\leftarrow H}$ denotes the operator norm on  $H$.
\end{assumption}

Obviously,  the Laplacian  operator $A=-\Delta$ defined in  a bounded domain $\Omega\subset \mathbb{R}^{d}$ satisfies Assumption \ref{assumption2.1}. Under this assumption, the  operator  $-A$ generates  an  analytic semigroup $\{e^{-tA}\}_{t\geq0}$ on $H$. By the sectoriality of
$A$,   the semigroup $\{e^{-tA}\}_{t\geq0}$  is  uniformly bounded on $ 0 \leq t \leq T$,
\begin{equation*}
\|e^{-tA}\|_{H\leftarrow H} +\| t^{\gamma}A^{\gamma}e^{-tA}\|_{H\leftarrow H} \leq C,
\end{equation*}
with constants $C$ and $\gamma$. This further leads to the uniformly bounded operators  $\varphi_{j}(-tA)$ defined on $H$.

 \section{A posteriori error estimates for linear problems} \label{sec3}

We first consider the E-Adams2 method with order two for  linear parabolic differential  equations,
\begin{equation}\label{linear problem}\left\{
\begin{array}{l}
u'(t)+Au(t)=f(t),\quad 0\leq t \leq T,\cr\noalign{\vskip2truemm} u(0)=u^0,
\end{array}
\right.\end{equation}
with $f:[0,T]\rightarrow H$.  For given  $\{v_{n}\}_{n=0}^N$, we use  the following notation
     \begin{equation*}
       \bar{\partial}v^{n}:= \frac{v^{n}-v^{n-1}}{k_{n}}, \quad v^{n-\frac{1}{2}}:=\frac{v^{n}+v^{n-1}}{2}, \quad n=1,\ldots,N.
     \end{equation*}
For linear problems,  the trapezoidal method  is employed to compute the approximation   $U^1\in D(A)$  of the value $u(t^{1})$,   i.e.,
\begin{equation}\label{trapezoidal}
	\bar{\partial}U^{1}+AU^{\frac{1}{2}}=f^{\frac{1}{2}},
\end{equation}
with  $U^{0} \in D(A)$ approximating $u^0$.  Note that even for $U^0 \in H$, the approximation  $U^{1}$ belongs to  $D(A)$  due to the regularization property of the numerical scheme \eqref{trapezoidal}. E-Adams2 nodal approximations  $U^{m}\in D(A)$ to the values $u(t^{m})$ are
    \begin{equation}\label{linear method}
    \begin{aligned}
        U^{n}&=e^{-k_{n}A}U^{n-1}+k_{n}\varphi_1(-k_{n}A)f^{n-1}+k_{n}r_{n}\varphi_2(-k_{n}A)\big(f^{n-1}-f^{n-2}\big) \cr\noalign{\vskip1truemm}
        &=U^{n-1}+k_{n}\varphi_1(-k_{n}A)\big(f^{n-1}-AU^{n-1}\big)+k_{n}r_{n}\varphi_2(-k_{n}A)\big(f^{n-1}-f^{n-2}\big),
    \end{aligned}
    \end{equation}
    with $f^{n-1}:=f(t^{n-1})$, $n=2,\ldots,N$.  By using these notation, the  method \eqref{linear method} can be  rewritten  as
     \begin{equation*}
        \bar{\partial} U^{n}=\varphi_1(-k_{n}A)\big(f^{n-1}-AU^{n-1}\big)+r_{n}\varphi_2(-k_{n}A)\big(f^{n-1}-f^{n-2}\big), \quad n=2,\ldots,N.
    \end{equation*}
     \begin{remark}[The computation of $U^1$]
    	Obviously, it is more straightforward and convenient to compute $U^1$ with an exponential method than the trapezoidal method when the E-Adams2  method is used to approximate $U^n$, $n\geq 2$. However, we note that the order of  exponential midpoint method  will be reduced in
    	certain norms, particularly in the $L^2(0,T;V)$ norm (see, e.g., \cite{Hochbruck2005b,Hu2024}), which subsequently affects the convergence rate
    	of the E-Adams2 method. Since the  research on a posteriori error estimates for other  exponential-type integrators remains notably underdeveloped and the optimal a posteriori error estimates for the trapezoidal or CN method for parabolic equations have been established in \cite{Akrivis2006}, in this work, we use the trapezoidal method to obtain $U^1$.
    \end{remark}
    
\subsection{E-Adams2 linear approximation and suboptimal order a posteriori error estimate}
Since the error  $u(t^n)-U^{n}$ is of second order, it is natural to introduce the continuous approximation $U:[0,T]\rightarrow D(A)$ to $u$ defined by the piecewise linear interpolant at the nodal values $U^{n-1}$ and $U^{n}$,
\begin{equation}\label{linear problem linear interpolant}
	U(t)=U^{n-1}+(t-t^{n-1})\bar{\partial} U^{n}, \quad t\in J_n,\quad  n=1,\ldots,N.
\end{equation}
The \emph{residual} $R(t) \in H$ of $U$ is defined by
\begin{equation*}
	R(t):=U^{\prime}(t)+AU(t)-f(t), \quad t\in J_n,  \quad  n=1,\ldots,N,
\end{equation*}
which can be viewed as the amount by which the approximation solution $U$  misses satisfying \eqref{linear problem}.  From \eqref{linear problem linear interpolant}, the residual $R(t)$ takes the form
\begin{equation}\label{residual1}
	\begin{aligned}
		R(t)
		&=(\varphi_1(-k_{n}A)-I)\big(f^{n-1}-AU^{n-1}\big)+r_{n}\varphi_2(-k_{n}A)\big(f^{n-1}-f^{n-2}\big)
		\cr\noalign{\vskip1.5truemm}
		&\quad +(t-t^{n-1})A\bar{\partial} U^{n}+f^{n-1}-f(t), \quad t\in J_n, \quad n\geq2.
	\end{aligned}
\end{equation}
Suppose that  $U^1$ is computed by  the trapezoidal method (a variant of the CN method; see \cite{Akrivis2006,Akrivis2010}),  then  the residual  $R(t)\in H$ of $U$ is
\begin{equation*}
	R(t)=(t-t^{\frac{1}{2}})A\bar{\partial}U^{1}+f^{\frac{1}{2}}-f(t),  \quad t \in J_1.
\end{equation*}
Akrivis et al. \cite{Akrivis2006} have presented the first order residual $R(t)$ and introduced a  piecewise quadratic  reconstruction $\hat{U}$ of $U$ to recover the optimal order.

Under Assumption $2.1$, we define the following graph norm (see, e.g., \cite{Pazy1983,Hochbruck2010})
\begin{equation*}
	\|x\|_{D(A)}:=\|Ax\|_{H}+\|x\|_{H}.
\end{equation*}
According to the  recurrence relation \eqref{recurrence}, the following estimate holds for  a  constant $M_{0}$,
\begin{equation*}
	\begin{aligned}
		\|\varphi_1(-kA)-I\|_{H\leftarrow D(A)}:&=\sup_{\substack{x \in D(A) \\ x \neq 0}} \frac{\|(\varphi_1(-kA)-I)x\|_H}{\|x\|_{D(A)}}
		\leq \sup_{\substack{x \in D(A) \\ x \neq 0}}  \frac{M_{0}}{2}k\frac{\|Ax\|_{H}}{\|x\|_{D(A)}},
	\end{aligned}
\end{equation*}
which implies that  \eqref{residual1} has first order accuracy. We denote the
error by $e:=u-U$. Then the error $e$ satisfies the error equation $e^{\prime}+Ae=-R$. An application of  the energy technique to the error equation yields  suboptimal error bounds.
\subsection{Optimal order a  posteriori error estimates for linear problems}
To recover the optimal order,  this subsection  introduces a reconstruction $\hat{U}$ of $U$ that differs from the single-step case (see \cite{Hu2024,Hu2026}) by specifically handling the $\varphi_2$-term in \eqref{nonlinear method}, and then derives optimal order residual-based a posteriori error estimates for the E-Adams2 method for linear problems.

\subsubsection{E-Adams2  reconstruction}\label{sec3.1}
 Normally, the reconstruction $\hat{U}$ of $U$ shall possess the following  fundamental properties: (i) it should be continuous with $U$ at the nodes $t^{n}, n=0,\ldots,N$, and (ii) it provides the second order residual such that  residual-based a posteriori error estimates are of second (optimal) order, and (iii) the a posteriori error estimator derived from it is easy to compute.

To ensure that the reconstruction $\hat{U}$ of $U$ retains the fundamental properties mentioned above, based on the analysis for the approximation $U$ and its first order residual, let us first define  the  function $\psi(t): J_{n} \rightarrow H$,
\begin{equation*}
\begin{aligned}
\psi(t)&:=\varphi_1(-k_{n}A)f^{n-1}+r_{n}\varphi_2(-k_{n}A)(f^{n-1}-f^{n-2})+(t-t^{n-\frac{1}{2}})\varphi_1(-k_{n}A)\bar{\partial}f^{n-1},
\end{aligned}
\end{equation*}
for $t \in J_{n}$, $n\geq2$, and denote the piecewise quadratic polynomials $\Psi$  by $\Psi(t):=\int_{t^{n-1}}^{t}\psi(s)ds$. A simple calculation gives
\begin{equation*}
	\begin{aligned}
\Psi(t)&= (t-t^{n-1})\Big[\varphi_1(-k_{n}A)f^{n-1}+r_{n}\varphi_2(-k_{n}A)(f^{n-1}-f^{n-2})\Big]\cr\noalign{\vskip1.5truemm}
&\quad +\frac{1}{2}(t-t^{n-1})(t-t^{n})\varphi_1(-k_{n}A)\bar{\partial}f^{n-1}, \quad t\in J_{n}, \quad n\geq2.
\end{aligned}
\end{equation*}
 It is straightforward to show that $\Psi(t^{n-1})=0$, and
\begin{equation*}
 \quad \Psi(t^{n})=k_{n}\varphi_1(-k_{n}A)f^{n-1}+k_nr_{n}\varphi_2(-k_{n}A)(f^{n-1}-f^{n-2})=\int_{J_{n}} \psi(s)ds.
\end{equation*}
For any $t \in J_{n}$, $n\geq2$, we are  now  in a position to introduce the \emph{E-Adams2 reconstruction} $\hat{U}$ of $U$ by:
\begin{equation}\label{reconstruction}
\begin{aligned}
\hat{U}(t)&:=U^{n-1}-\int_{t^{n-1}}^{t} \varphi_1(-k_{n}A)AUds+\Psi(t)+\int_{t^{n-1}}^{t} \big(I-\varphi_1(-k_{n}A)\big)U^{\prime}(s)ds\cr\noalign{\vskip1.5truemm}
&=U^{n-1}-\int_{t^{n-1}}^{t} \varphi_1(-k_{n}A)AUds+\Psi(t) +(t-t^{n-1})\big(I-\varphi_1(-k_{n}A)\big)\bar{\partial}U^{n}.
\end{aligned}
\end{equation}
It should be noted that the term $k_{n}\varphi_1(-k_{n}A)AU^{n-1}$ appeared in \eqref{linear method}, approximating the integral in \eqref{reconstruction},  is of first order.  Based on this fact, we  evaluate the integral in \eqref{reconstruction}  by  the  first order approximation $(t-t^{n-1})\varphi_1(-k_{n}A)A\big(c_1U^{n-1}+c_2U^{n}\big)$ with the parameters  $c_1=-k_{n}^{-1}(\varphi_1(-k_{n}A)A)^{-1}\big[I-\varphi_1(-k_{n}A)-k_{n}\varphi_1(-k_{n}A)A\big]$, and
 $c_2=k_{n}^{-1}(\varphi_1(-k_{n}A)A)^{-1}\big[I-\varphi_1(-k_{n}A)\big]$; then  we obtain
 {\small{
\begin{equation}\label{reconstruction2}
\begin{aligned}
\hat{U}(t)=\ &U^{n-1}+\frac{t-t^{n-1}}{k_{n}}\Big[\big(I-\varphi_1(-k_{n}A)-k_{n}\varphi_1(-k_{n}A)A\big)U^{n-1}
-\big(I-\varphi_1(-k_{n}A)\big)U^{n}\Big]\cr\noalign{\vskip2truemm}
&+\Psi(t)+\frac{t-t^{n-1}}{k_{n}}\Big[I-\varphi_1(-k_{n}A)\Big](U^{n}-U^{n-1})\cr\noalign{\vskip1truemm}
=\ &U^{n-1}+(t-t^{n-1})\Big[\varphi_1(-k_{n}A)(f^{n-1}-AU^{n-1})+r_{n}\varphi_2(-k_{n}A)(f^{n-1}-f^{n-2})\Big]\cr\noalign{\vskip2truemm}
&+\frac{1}{2}(t-t^{n-1})(t-t^{n})\varphi_1(-k_{n}A)\bar{\partial}f^{n-1},
\end{aligned}
\end{equation}}}
for any $t \in J_{n}$, $n\geq2$.
From \eqref{reconstruction2}, it is simple to verify that $\hat{U}(t^{n-1})=U^{n-1}$, and
\begin{equation*}
\hat{U}(t^{n})=U^{n-1}+k_{n}\varphi_1(-k_{n}A)\big(f^{n-1}-AU^{n-1}\big)+k_{n}r_{n}\varphi_2(-k_{n}A)\big(f^{n-1}-f^{n-2}\big)=U^{n}.
\end{equation*}
As a consequence, $\hat{U}$ coincides with $U$ at the nodes $t^{1},\ldots, t^{n}$. Since  we take the following  reconstruction $\hat{U}$ of $U$ over the interval $J_1$, $\hat{U}$ coincides with $U$ at the nodes $t^{0}$ and $t^{1}$; in particular,  $\hat{U}:[0,T]\rightarrow H$ is continuous.

Once $U^{1}$ is computed by the trapezoidal method, a piecewise quadratic reconstruction $\hat{U}$, introduced by \cite{Akrivis2006}, can be defined as 
\begin{equation*}
	\hat{U}(t)=U^0-\int_{t^0}^{t}AU(s)ds+\int_{t^0}^{t}\Big(f^{\frac{1}{2}}+(s-t^\frac{1}{2})\bar{\partial} f^1\Big)ds, \quad \forall t \in J_1.
\end{equation*}
It then follows that 
\begin{equation}\label{derivative of recon2}
	\hat{U}^{\prime}(t)+AU(t)=f^{\frac{1}{2}}+(t-t^{\frac{1}{2}})\bar{\partial}f^{1},   \quad \forall t\in J_1.
\end{equation}
The  residual $\hat{R}(t)\in H$ of $\hat{U}$ is defined by
\begin{equation*}
	\hat{R}(t):=\hat{U}^{\prime}(t)+A\hat{U}(t)-f(t), \quad  t \in J_{n}, \quad n=1,\ldots,N.
\end{equation*}
For $t\in J_{1}$,  the second order residual $\hat{R}(t)$ of $\hat{U}$ is given by
\begin{equation}\label{recon error2}
	\hat{R}(t)=A(\hat{U}(t)-U(t))+R_{f_{1}}(t),
\end{equation}
where
\begin{equation}\label{f1}
	R_{f_1}(t)=f^{\frac{1}{2}}+(t-t^{\frac{1}{2}})\bar{\partial}f^1-f(t), \quad  t \in J_1.
\end{equation}

From \eqref{reconstruction}, the time derivative of $\hat{U}$ satisfies
\begin{equation}\label{derivative of recon}
\begin{aligned}
\hat{U}^{\prime}(t)&+\varphi_1(-k_{n}A)AU(t)=\psi(t)+\big(I-\varphi_1(-k_{n}A)\big)U^{\prime}(t), \quad \forall t \in J_n, \quad n \geq 2.
\end{aligned}
\end{equation}
Therefore, for  $t \in J_n$, the residual $\hat{R}(t)$ can be written in the form
\begin{equation}\label{optimal residual}
\begin{aligned}
\hat{R}(t)=&A\hat{U}(t)-\varphi_1(-k_{n}A)AU(t)+\psi(t)+\big(I-\varphi_1(-k_{n}A)\big)U^{\prime}(t)-f(t)\cr\noalign{\vskip1.5truemm}
=&A(\hat{U}(t)-U(t))+(I-\varphi_1(-k_{n}A))\big(U^{\prime}(t)+AU(t)-f(t)\big)+\varphi_1(-k_{n}A)(f^{n-1}\cr\noalign{\vskip1.5truemm}
&+(t-t^{n-1})\bar{\partial}f^{n-1}-f(t))+r_{n}\Big[\varphi_2(-k_{n}A)-\frac{1}{2}\varphi_1(-k_{n}A)\Big]\big(f^{n-1}-f^{n-2}\big).
\end{aligned}
\end{equation}
 The second order convergence rate of the residual $\hat{R}$  is analyzed   by the following remarks.

 \begin{remark}[Estimate of  $\hat{U}-U$]
  Subtracting \eqref{linear problem linear interpolant} from \eqref{reconstruction2}, we get
\begin{equation*}
\begin{aligned}
\hat{U}(t)-U(t)=&(t-t^{n-1})\Big[\varphi_1(-k_{n}A)(f^{n-1}-AU^{n-1})+r_{n}\varphi_2(-k_{n}A)(f^{n-1}-f^{n-2})\cr\noalign{\vskip0.5truemm}
&+\frac{1}{2}(t-t^{n})\varphi_1(-k_{n}A)\bar{\partial}f^{n-1}-\bar{\partial}U^{n}\Big]\cr\noalign{\vskip0.5truemm}
=&\frac{1}{2}(t-t^{n-1})(t-t^{n})\varphi_1(-k_{n}A)\bar{\partial}f^{n-1},
\end{aligned}
\end{equation*}
for any $ t\in J_{n}$, $n \geq 2$, and it is easy to verify that
\begin{equation}\label{formula1}
\begin{aligned}
 \int_{J_{n}}\|\hat{U}(t)-U(t)\|^2dt&=\frac{1}{4}\int_{J_{n}}(t-t^{n-1})^2(t-t^{n})^2\|\varphi_1(-k_{n}A)\bar{\partial}f^{n-1}\|^2dt\cr\noalign{\vskip1.5truemm}
&=\frac{k_{n}^5}{120}\|\varphi_1(-k_{n}A)\bar{\partial}f^{n-1}\|^2, \quad \forall t\in J_{n}, \quad n \geq 2.
\end{aligned}
\end{equation}
In \cite{Akrivis2006}, the difference between $\hat{U}$ and $U$ was given by
\begin{equation*}
	\hat{U}(t)-U(t)=\frac{1}{2}(t-t^{0})(t-t^{1})\bar{\partial}(f^{1}-AU^{1}), \quad \forall t \in J_1.
\end{equation*}
A  direct calculation gives
\begin{equation*}
	\int_{J_1}\|\hat{U}(t)-U(t)\|^2dt=\frac{k_{1}^5}{120}\|\bar{\partial}(f^{1}-AU^{1})\|^2, \quad \forall t \in J_1.
\end{equation*}
 Therefore, the a posteriori quantity  $\epsilon_U:=\int_{0}^{T} \|\hat{U}(t)-U(t)\|^2dt$ is  computed  as
\begin{equation*}
	\begin{aligned}
		\epsilon_U&=\int_{J_1}\big\|\hat{U}(t)-U(t)\big\|^2dt+\sum_{n=2}^{N}  \int_{J_{n}}\big\|\hat{U}(t)-U(t)\big\|^2dt\cr\noalign{\vskip0.2truemm}
		&=\frac{k_{1}^5}{120}\big\|\bar{\partial}(f^{1}-AU^{1})\big\|^2+\sum_{n=2}^{N}\frac{k_{n}^5}{120}\big\|\varphi_1(-k_{n}A)\bar{\partial}f^{n-1}
		\big\|^2.
	\end{aligned}
\end{equation*}

\end{remark}
\begin{remark}[Optimality of the residual $\hat{R}$]
 By formula \eqref{formula1} and  the bounded  operators $\varphi_{j}(-kA)$,  the a posteriori quantities $\hat{U}(t)-U(t)$ and $\varphi_1(-k_{n}A)\big(f^{n-1}+(t-t^{n-1})\bar{\partial}f^{n-1}-f(t)\big)$ are of second order   when $f \in C^2(0,T;H)$. Using  the relations $\varphi_1(-k_{n}A)-I=\mathcal{O}(k_{n})$ and $\frac{1}{2}\varphi_1(-k_{n}A)-\varphi_2(-k_{n}A)=\mathcal{O}(k_{n})$, the  a posteriori quantities   $(\varphi_1(-k_{n}A)-I)R(t)$  and $r_{n}\big[\frac{1}{2}\varphi_1(-k_{n}A)-\varphi_2(-k_{n}A)\big](f^{n-1}-f^{n-2})$ are  of  second order. Therefore,  \eqref{optimal residual} is of optimal (second) order.
\end{remark}

\begin{remark}
Note that  the second order a posteriori quantities $\big(\varphi_1(-k_{n}A)-I\big)R(t)$ and   $r_{n}\big[\frac{1}{2}\varphi_1(-k_{n}A)-\varphi_2(-k_{n}A)\big](f^{n-1}-f^{n-2})$ are related to the relations $\varphi_1(-k_{n}A)-I=\mathcal{O}(k_{n})$ and $\frac{1}{2}\varphi_1(-k_{n}A)-\varphi_2(-k_{n}A)=\mathcal{O}(k_{n})$. The practical computation of these quantities may affect the  convergence rates  of the a posteriori estimators, particularly for larger temporal step-sizes.  The reduced-order behavior of the a posteriori quantities may occur in semilinear problems, and the convergence rates of all these quantities will be confirmed by numerical examples  in Section \ref{sec6}.
\end{remark}

\subsubsection{Optimal order a  posteriori error analysis}\label{sec3.2}
To begin with, we introduce some notations. Let $V:=D(A^{\frac{1}{2}})$, and denote
by $|\cdot|$, $\|\cdot\|$ the norms in $H$ and $V$, with $\|v\|:=|A^{\frac{1}{2}}v|=(Av,v)^{\frac{1}{2}}$  for any $v\in V$.  We denote the dual of $V$ by
$V^{\star}$ $(V\subset H\subset V^{\star})$ with the dual norm $\|v\|_{\star}:=|A^{-\frac{1}{2}}v|=(v,A^{-1}v)^{\frac{1}{2}}$  for any $v \in V^{\star}$. We still denote
by $(\cdot,\cdot)$ the dual pairing  between $V^{\star}$ and $V$;  by  $L^{p}(J;X)$, $1\leq p\leq \infty$,  the standard  Lebesgue spaces  on the time  interval $J$  and a  Banach space $X$ ($X=H$, $V$ or  $V^{\star}$)  with the  norm $\|\cdot\|_{L^{p}(J;X)}$.  We further introduce the following regularity condition for  $\hat{U}$ defined in \eqref{reconstruction}:
\begin{equation*}
  \hat{U} \in V, \quad \forall t \in [0,T].
\end{equation*}
 Let the errors $e$ and $\hat{e}$, $e:=u-U$ and $\hat{e}:=u-\hat{U}$; subtracting  \eqref{derivative of recon} from  \eqref{linear problem}, we get  the following error equation
\begin{equation}\label{error1}
\begin{aligned}
\displaystyle \hat{e}^{\prime}(t)+Ae(t)&=f(t)-\big(I-\varphi_1(-k_{n}A)\big)(U^{\prime}(t)+AU(t))-\psi(t)\cr\noalign{\vskip0.5truemm}
&=\big(\varphi_1(-k_{n}A)-I\big)R(t)+\varphi_1(-k_{n}A)R_{f}(t)\cr\noalign{\vskip0.5truemm}
&\quad+r_{n}\Big[\frac{1}{2}\varphi_1(-k_{n}A)-\varphi_2(-k_{n}A)\Big]\big(f^{n-1}-f^{n-2}\big), \quad t\in J_{n}, \quad n\geq2,
\end{aligned}
\end{equation}
  where $R(t)$ is the residual of $U$,  and $R_f(t)$ is  defined by
\begin{equation}\label{linear interpolation of f}
  R_{f}(t):=f(t)-f^{n-1}-(t-t^{n-1})\bar{\partial}f^{n-1},  \quad t \in J_n, \quad n\geq2.
\end{equation}
We take in \eqref{error1} the inner product with $\hat{e}(t)$ and arrive at
\begin{equation*}
\begin{aligned}
 \big(\hat{e}^{\prime}(t),\hat{e}(t)\big)+\big(Ae(t),\hat{e}(t)\big)
&=\Big((\varphi_1(-k_{n}A)-I)R(t),\hat{e}(t)\Big)+\big(\varphi_1(-k_{n}A)R_{f}(t),\hat{e}(t)\big)\\
&\ +\Big(r_{n}\Big[\frac{1}{2}\varphi_1(-k_{n}A)-\varphi_2(-k_{n}A)\Big]\big(f^{n-1}-f^{n-2}\big),\hat{e}(t)\Big).
\end{aligned}
\end{equation*}

Using the relations $
\big(Ae(t),\hat{e}(t)\big)=\frac{1}{2}\big(\|e(t)\|^2+\|\hat{e}(t)\|^2-\|\hat{e}(t)-e(t)\|^2 \big)$ and $\hat{e}(t)-e(t)=U(t)-\hat{U}(t)$,
we have
\begin{equation}\label{error2}
\begin{aligned}
 & \frac{d}{dt}|\hat{e}(t)|^2+\|e(t)\|^2+\|\hat{e}(t)\|^2\cr\noalign{\vskip0.5truemm}
\displaystyle &=\|\hat{U}(t)-U(t)\|^2+2\Big(\big(\varphi_1(-k_{n}A)-I\big)R(t),\hat{e}(t)\Big)+2\big(\varphi_1(-k_{n}A)R_{f}(t),\hat{e}(t)\big)\cr\noalign{\vskip0.5truemm}
\displaystyle  &\quad +2\Big(r_{n}\Big[\frac{1}{2}\varphi_1(-k_{n}A)-\varphi_2(-k_{n}A)\Big]\big(f^{n-1}-f^{n-2}\big),\hat{e}(t)\Big).
\end{aligned}
\end{equation}

{\bf $L^2(0,t;V)$-estimate.} The following theorem shows an $L^2(0,T;V)$-norm a posteriori error estimates for the method  \eqref{linear method} for linear problems.
\begin{theorem}\label{linear problem error estimates}
Let $U(t)$ be the continuous approximation defined in \eqref{linear problem linear interpolant}, and $\hat{U}(t)$ be the E-Adams2 reconstruction of
$U(t)$ defined by \eqref{reconstruction}. Denote the  errors  by  $e=u-U$ and $\hat{e}=u-\hat{U}$. Then the following a posteriori error bounds are valid for $t\in J_{n}$, $n\geq2$,and $\epsilon_1>3/2$:
\begin{equation}\label{linear error}
  \begin{aligned}
 &\frac{2\epsilon_1-3}{4\epsilon_1-3} \int_{0}^{t} \|\hat{U}(s)-U(s)\|^2ds \leq |\hat{e}(t)|^2+\int_{0}^{t} \Big(\|e(s)\|^2+\big(1-\frac{3}{2\epsilon_1}\big)\|\hat{e}(s)\|^2\Big)ds\\
  &  \leq |\hat{e}(0)|^2+\int_{0}^{t} \|\hat{U}(s)-U(s)\|^2ds
  +\frac{2\epsilon_1}{3}\int_{t^0}^{t^{1}} \|R_{f_1}(s)\|_{\star}^2ds\\
  &\quad +2\epsilon_1\sum_{i=2}^{n-1} \int_{t^{i-1}}^{t^{i}}   \bigg(\big\|(\varphi_1(-k_{i}A)-I)  R(s)\big\|_{\star}^2    +\big\|\varphi_1(-k_{i}A)R_f(s)\big\|_{\star}^2\\
 &\quad+\Big\|r_{i}\Big[\frac{1}{2}\varphi_1(-k_{i}A)-\varphi_2(-k_{i}A)\Big] \big(f^{i-1} -f^{i-2}\big)\Big\|_{\star}^2 \bigg)ds\\
    &\quad +2\epsilon_1\int_{t^{n-1}}^{t} \bigg( \big\|(\varphi_1(-k_{n}A)-I) R(s)\big\|_{\star}^2 +\big\|\varphi_1(-k_{n}A)R_f(s)\big\|_{\star}^2 \\
  &\quad +\Big\|r_{n}\Big[\frac{1}{2}\varphi_1(-k_{n}A)-\varphi_2(-k_{n}A)\Big]\big(f^{n-1}-f^{n-2}\big)\Big\|_{\star}^2 \bigg)ds,
  \end{aligned}
\end{equation}
 where $R(s)$, $R_{f_{1}}(s)$, and  $R_{f}(s)$  are defined by \eqref{residual1}, \eqref{f1}, and \eqref{linear interpolation of f}, respectively.
 \end{theorem}
\begin{proof}
  We apply the  Cauchy--Schwarz inequality and the Young's inequality to \eqref{error2}, and obtain
   \begin{equation}\label{error3}
   \begin{aligned}
     & \frac{d}{dt}|\hat{e}(t)|^2+\|e(t)\|^2+\big(1-\frac{3}{2\epsilon_1}\big)\|\hat{e}(t)\|^2 \\
      &\leq \|\hat{U}(t)-U(t)\|^2+2\epsilon_1\bigg(\big\|\big(\varphi_1(-k_{n}A)-I\big)R(t)\big\|_{\star}^2+\big\|\varphi_1(-k_{n}A)R_f(t)\big\|_{\star}^2\\
     &\quad +\Big\|r_{n}\Big[\frac{1}{2}\varphi_1(-k_{n}A)-\varphi_2(-k_{n}A)\Big]\big(f^{n-1}-f^{n-2}\big)\Big\|_{\star}^2\bigg),
  \end{aligned}
   \end{equation}
for $t\in J_n$, $n\geq 2$, and $\epsilon_1>3/2$. Integrating inequality \eqref{error3} from $t^1$ to $t\in J_n$, one has
   \begin{equation}\label{estimate1}
   \begin{aligned}
 &  |\hat{e}(t)|^2+\int_{t^1}^{t} \Big(\|e(s)\|^2+\big(1-\frac{3}{2\epsilon_1}\big)\|\hat{e}(s)\|^2\Big)ds\\
  & \leq |\hat{e}(t^1)|^2+ \int_{t^1}^{t} \|\hat{U}(s)-U(s)\|^2ds +2\epsilon_1\sum_{i=2}^{n-1} \int_{t^{i-1}}^{t^{i}}  \bigg( \big\|\big(\varphi_1(-k_{i}A)-I\big) R(s)\big\|_{\star}^2\\
 &\quad  +\|\varphi_1(-k_{i}A)R_f(s)\|_{\star}^2 + \Big\|r_{i}\Big[\frac{1}{2}\varphi_1(-k_{i}A)-\varphi_2(-k_{i}A)\Big]  \big(f^{i-1}-f^{i-2}\big)\Big\|_{\star}^2  \bigg)ds \\
 & \quad +2\epsilon_1 \int_{t^{n-1}}^{t} \bigg(\big\|\big(\varphi_1(-k_{n}A)-I\big) R(s) \big \|_{\star}^2+\big\|\varphi_1(-k_{n}A)R_f(s)\big\|_{\star}^2\\
 &\quad  +\Big\|r_{n}	\Big[\frac{1}{2}\varphi_1(-k_{n}A)-\varphi_2(-k_{n}A)\Big] \big(f^{n-1}-f^{n-2}\big)\Big\|_{\star}^2\bigg) ds.
\end{aligned}
   \end{equation}
  Now that  $U^{1}$ is computed  by  the trapezoidal method, it follows that
\begin{equation*}
	\begin{aligned}
		\frac{d}{dt}|&\hat{e}(t)|^2+\|e(t)\|^2+\|\hat{e}(t)\|^2
		=\|\hat{U}(t)-U(t)\|^2+2\big(R_{f_1}(t),\hat{e}(t)\big), \quad t \in J_1.
	\end{aligned}
\end{equation*}
Using a similar argument for  deriving \eqref{estimate1}, the following a posteriori error estimate holds   for  the errors $e$ and $\hat{e}$ over the interval $J_1$:
\begin{equation}\label{estimate2}
	\begin{aligned}
		|\hat{e}(t^1)|^2&+\int_{t^0}^{t^1} \Big(\|e(s)\|^2+\big(1-\frac{3}{2\epsilon_1}\big)\|\hat{e}(s)\|^2\Big)ds \\
		&\leq  |\hat{e}(0)|^2+\int_{t^0}^{t^1} \|\hat{U}(s)-U(s)\|^2ds+\frac{2\epsilon_1}{3}\int_{t^0}^{t^1} \|R_{f_1}(s)\|_{\star}^2ds.
	\end{aligned}
\end{equation}
Summing  \eqref{estimate1} and \eqref{estimate2},  we obtain the desired upper bound.

Combining the relation
\begin{equation*}
\|\hat{U}(t)-U(t)\|\leq\|e(t)\|+\|\hat{e}(t)\|,
\end{equation*} and the Young's inequality, we infer that
   \begin{equation}\label{lower error}
     \frac{2\epsilon_1-3}{4\epsilon_1-3}\|\hat{U}(t)-U(t)\|^2\leq\|e(t)\|^2+\big(1-\frac{3}{2\epsilon_1}\big)\|\hat{e}(t)\|^2;
   \end{equation}
   integrating inequality \eqref{lower error} from $0$ to $t$, the lower bound can be easily  obtained, and thus completes the proof
   of Theorem \ref{linear problem error estimates}.
\end{proof}

{\bf $L^{\infty}(0,t;H)$-estimate.} In view of \eqref{linear error},  the following a posteriori error estimate in the
$L^{\infty}(0,t;H)$-norm is valid for $t\in J_{n}$, $n\geq2$, and $\epsilon_1>3/2$:
\begin{equation}\label{linear error2}
  \begin{aligned}
 \max_{0\leq s \leq t}   |\hat{e}(s)|^2
&  \leq|\hat{e}(0)|^2+\int_{0}^{t} \|\hat{U}(s)-U(s)\|^2ds +\frac{2\epsilon_1}{3}\int_{t^0}^{t^{1}} \|R_{f_1}(s)\|_{\star}^2ds\\
&\quad +2\epsilon_1\sum_{i=2}^{n-1} \int_{t^{i-1}}^{t^{i}}   \bigg(\big\|\big(\varphi_1(-k_{i}A)-I\big)R(s)\big\|_{\star}^2 
  +\|\varphi_1(-k_{i}A)R_f(s)\|_{\star}^2 \\
& \quad + \Big\|r_{i}\Big[\frac{1}{2}\varphi_1(-k_{i}A)-\varphi_2(-k_{i}A)\Big] \big(f^{i-1}-f^{i-2}\big)\Big\|_{\star}^2 \bigg)ds\\
  &\quad +2\epsilon_1\int_{t^{n-1}}^{t} \bigg( \big\|\big(\varphi_1(-k_{n}A)-I\big) R(s)\big\|_{\star}^2 +\big\|\varphi_1(-k_{n}A)R_f(s)\big\|_{\star}^2  \\
  &\quad +\Big\|r_{n}\Big[\frac{1}{2}\varphi_1(-k_{n}A)-\varphi_2(-k_{n}A)\Big]\big(f^{n-1}-f^{n-2}\big)\Big\|_{\star}^2 \bigg)ds.
  \end{aligned}
\end{equation}
As shown in  Section \ref{sec3},  the reconstruction $\hat{U}$ coincides with $U$ at  the nodes $t^{0},\ldots, t^{N}$. Hence,  we have the following error bound:
\begin{equation}\label{discrete estimates for linear}
  \begin{aligned}
\max_{0\leq n \leq  N}  | e(t^{n})|^2
 & \leq |e(0)|^2+\int_{0}^{t^N} \|\hat{U}(t)-U(t)\|^2dt +\frac{2\epsilon_1}{3}\int_{t^0}^{t^{1}} \|R_{f_1}(s)\|_{\star}^2ds \\
&\quad +2\epsilon_1\sum_{n=2}^{N} \int_{t^{n-1}}^{t^{n}}   \bigg( \big\|\big(\varphi_1(-k_{n}A)-I\big)R(t)\big\|_{\star}^2+\|\varphi_1(-k_{n}A)R_f(t)\|_{\star}^2
  \\
 & \quad  + \Big \|r_{n}\big[\frac{1}{2}\varphi_1(-k_{n}A)-\varphi_2(-k_{n}A)\big] \big(f^{n-1}-f^{n-2}\big)\Big \|_{\star}^2 \bigg)dt.
  \end{aligned}
\end{equation}

\begin{remark}
 Theorem \ref{linear problem error estimates}  and formula \eqref{linear error2} show optimal order  error bounds of the method \eqref{linear method} for linear problems in the $L^2(0,t;V)$- and $L^{\infty}(0,t;H)$-norms. Error bounds depend only on the data of the problem and numerical solutions; therefore they are computable.
\end{remark}
\begin{remark}[Choice of $\epsilon_1$] The same parameter $\epsilon_1>0$ appears in both the upper and lower bounds as a direct consequence of the unified application of the Young's inequality in the derivation. To ensure the rigorous validity of the error estimate,  the condition $\epsilon_1>3/2$ is required, which guarantees both $\frac{2\epsilon_1-3}{4\epsilon_1-3}>0$ and $1-\frac{3}{2\epsilon_1}>0$. However, for larger time step-sizes, the practical computation of the  $\varphi$-functions may slightly affect the
 convergence rates of all the a posteriori quantities. In particular, the ratio of the lower bound  to the  error $e(t)$  may exceed $1$ in various norms.
To avoid this,  the parameter $\epsilon_1$ can be set closer to $3/2$ in practice, such that the ratio does not exceed $1$.
\end{remark}

\section{A posteriori error estimates for semilinear problems}\label{sec4} In this section, we derive optimal order a posteriori error bounds for the method \eqref{nonlinear method} for a class of semilinear problems introduced in \cite[Section 6]{Akrivis2006}.

 Let us assume that $B(t,\cdot)$ is an operator from $V$ into $V^{\star}$, and satisfies the following
local one-sided Lipschitz condition:
\begin{equation}\label{one-sided Lip}
(B(t,v)-B(t,w), v-w)\leq \lambda \|v-w\|^2+\mu|v-w|^2, \quad \forall v,w \in T_{u},
\end{equation}
 where $T_{u}:= \{v\in V: \min_{t} \|u(t)-v\|\leq 1\}$, with   constants   $\lambda < 1$  and  $\mu$. Moreover, we assume that the following local Lipschitz condition for $B(t,\cdot)$ in a strip along the exact solution $u$ is valid:
\begin{equation}\label{Lip}
\|B(t,v)-B(t,w)\|_{\star} \leq L\|v-w\|, \quad \forall v,w \in T_{u},
\end{equation}
with a constant $L$.
\subsection{E-Adams2 method}\label{sec4.1}
Similar to the linear case, we compute  $U^1 \in D(A)$  by the trapezoidal method, i.e.,
 \begin{equation*}
 \bar{\partial}U^{1}+AU^{\frac{1}{2}}=B^{\frac{1}{2}},
 \end{equation*}
 with $U^{0}:=u^{0}\in D(A)$ approximating $u^0$, and E-Adams2 nodal approximations $U^{m}\in D(A)$ to the values  $u(t^{m})$ are
 \begin{equation*}
U^{n}=U^{n-1}+k_{n}\varphi_1(-k_{n}A)(B^{n-1}-AU^{n-1})+k_{n}r_{n}\varphi_2(-k_{n}A)\big(B^{n-1}-B^{n-2}\big),\quad n\geq2.
\end{equation*}
Let us define the continuous approximation $U:[0,T]\rightarrow D(A)$  to $u$ by the linear interpolant  of nodal values $U^{n-1}$ and $U^{n}$,
 \begin{equation}\label{semi linear interpolant}
 U(t)=U^{n-1}+(t-t^{n-1})\bar{\partial}U^{n}, \quad  t \in J_{n}, \quad  n=1,\ldots,N.
\end{equation}
The  residual   $R(t)\in H$ of $U$  is defined as  follows:
\begin{equation}\label{semi linear resi1}
\begin{aligned}
R(t):&=U^{\prime}(t)+AU(t)-B(t,U(t))\\
&=\big(\varphi_1(-k_{n}A)-I\big)(B^{n-1}-AU^{n-1})+r_{n}\varphi_2(-k_{n}A) \big(B^{n-1}-B^{n-2}\big)\\
&\quad+(t-t^{n-1})A\bar{\partial} U^{n}+B^{n-1}-B(t,U(t)),   \quad  t\in J_{n}, \quad  n\geq2.
\end{aligned}
\end{equation}
Due to the  relation  $\varphi_1(-k_{n}A)-I=\mathcal{O}(k_{n})$, it follows  that  \eqref{semi linear resi1} is  of first order, i.e., of suboptimal order.

    To establish error bounds that match the convergence order of the method \eqref{nonlinear method},  the following piecewise continuous quadratic reconstruction $\hat{U}$ is introduced. To begin with, let us define the functions $\bar{\psi}(t): J_{n} \rightarrow H$,
\begin{equation*}
\bar{\psi}(t):=\varphi_1(-k_{n}A)B^{n-1}+r_{n}\varphi_2(-k_{n}A)(B^{n-1}-B^{n-2})+(t-t^{n-\frac{1}{2}})\varphi_1(-k_{n}A)\bar{\partial}B^{n-1},
\end{equation*}
for $ t\in J_{n}$, $n\geq2.$
We recall that \eqref{reconstruction}, and let the  E-Adams2  reconstruction $\hat{U}$ of $U$ be:
\begin{equation}\label{semi reconstruction1}
\begin{aligned}
\hat{U}(t):&=U^{n-1}-\int_{t^{n-1}}^{t} \varphi_1(-k_{n}A)AU(s)ds+\int_{t^{n-1}}^{t} \bar{\psi}(s)ds \\
& \quad +\int_{t^{n-1}}^{t} \big(I-\varphi_1(-k_{n}A)\big)U^{\prime}(s)ds, \quad   \forall t \in J_{n}, \quad   n\geq2,
\end{aligned}
\end{equation}
i.e.,
\begin{equation*}
\begin{aligned}
\hat{U}(t)&=U^{n-1}+(t-t^{n-1})\Big[\varphi_1(-k_{n}A)(B^{n-1}-AU^{n-1})+r_{n}\varphi_2(-k_{n}A)(B^{n-1}\\
&\quad -B^{n-2})+\frac{1}{2}(t-t^{n})\varphi_1(-k_{n}A)\bar{\partial}B^{n-1}\Big],  \quad   \forall t \in J_{n}, \quad   n\geq2.
\end{aligned}
\end{equation*}
It is easy to show that: for any $t\in J_{n}$, $n\geq2$,
\begin{equation*}
\hat{U}(t)-U(t)=\frac{1}{2}(t-t^{n-1})(t-t^{n})\varphi_1(-k_{n}A)\bar{\partial}B^{n-1}.
\end{equation*}
In \cite{Akrivis2006}, a  piecewise quadratic reconstruction $\hat{U}$ was given by 
\begin{equation*}
\hat{U}(t)=U^{0}-\int_{t^0}^{t}AU(s)ds+\int_{t^0}^{t} \Big(B^{\frac{1}{2}}+\frac{2}{k_1}(s-t^{\frac{1}{2}})(B^{\frac{1}{2}}-B^{0})\Big)ds, \quad  \forall t \in J_1.
\end{equation*}

\subsection{A posteriori error estimates}\label{sec4.2}
Let us consider the errors $e$ and $\hat{e}$, $e:=u-U$ and $\hat{e}:=u-\hat{U}$, and assume that $\hat{U}(t)$, $U(t) \in T_{u}$ for all $ t \in [0,T]$.

Recalling that
\begin{equation*}\left\{
\begin{aligned}
&u^{\prime}(t)+Au(t)=B(t,u(t)),\\
&\hat{U}^{\prime}(t)+\varphi_1(-k_{n}A)AU(t)=\bar{\psi}(t)+(I-\varphi_1(-k_{n}A))U^{\prime}(t),
\end{aligned}
\right.\end{equation*}
for any $t \in J_{n}$, $n\geq2$, then we have
\begin{equation}\label{semi error equation}
\begin{aligned}
\hat{e}^{\prime}(t)+Ae(t)&=B(t,u(t))-B(t,U(t))+\big(\varphi_1(-k_{n}A)-I\big)R(t)+\varphi_1(-k_{n}A)R_{b}(t)\\
&\quad+r_{n}\Big[\frac{1}{2}\varphi_1(-k_{n}A)-\varphi_2(-k_{n}A)\Big]\big(B^{n-1}-B^{n-2}\big),
\end{aligned}
\end{equation}
where $R(t)$ is the residual of $U$, and  $R_b(t)$ is defined by
\begin{equation}\label{semi b}
R_b(t)=B(t,U(t))-B^{n-1}-(t-t^{n-1})\bar{\partial}B^{n-1},   \quad t \in J_{n}, \quad n \geq2.
\end{equation}
Taking in \eqref{semi error equation} the inner product with $\hat{e}(t)$, we obtain
\begin{equation}\label{semi error in1}
\begin{aligned}
\frac{d}{dt}|\hat{e}(t)|^2+\|e(t)\|^2+\|&\hat{e}(t)\|^2=\|\hat{U}(t)-U(t)\|^2+2\big(B(t,u(t))-B(t,U(t)),\hat{e}(t)\big) \\
& +2\Big(\big(\varphi_1(-k_{n}A)-I\big)R(t),\hat{e}(t)\Big)+2\big(\varphi_1(-k_{n}A)R_b(t),\hat{e}(t)\big)\\
&  +2\Big(r_{n}\Big[\frac{1}{2}\varphi_1(-k_{n}A)-\varphi_2(-k_{n}A)\Big]\big(B^{n-1}-B^{n-2}\big),\hat{e}(t)\Big).
\end{aligned}
\end{equation}

  Since $U^{1}$ is computed by the trapezoidal method, we deduce that
  \begin{equation*}
\begin{aligned}
\frac{d}{dt}|\hat{e}(t)|^2+\|e(t)\|^2+\|\hat{e}(t)\|^2=\ &\|\hat{U}(t)-U(t)\|^2+2\big(B(t,u(t))-B(t,U(t)),\hat{e}(t)\big)\\
&+2\big(R_{b_1}(t),\hat{e}(t)\big),  \quad   t\in J_1,
\end{aligned}
\end{equation*}
where
\begin{equation}\label{semi b_1}
R_{b_1}(t)=B(t,U(t))-B^{\frac{1}{2}}-\frac{2}{k_1}(t-t^{\frac{1}{2}})(B^{\frac{1}{2}}-B^0), \quad t\in J_1.
\end{equation}

 Now, we  derive optimal order a posteriori error estimates for the
 method  \eqref{nonlinear method} for a class of semilinear problems.

\begin{theorem}\label{semilinear results} Suppose that the nonlinear operator $B$ in \eqref{nonlinear problem} satisfies the conditions \eqref{one-sided Lip} and \eqref{Lip}, the continuous approximation $U(t)$ is defined by \eqref{semi linear interpolant}, and the E-Adams2 reconstruction $\hat{U}(t)$ of $U(t)$ is introduced by \eqref{semi reconstruction1}. Let the errors $e(t)$ and $\hat{e}(t)$ be defined by $e=u-U$ and $\hat{e}=u-\hat{U}$. Then the following  a  posteriori error estimate holds for $t\in J_{n}$, $n\geq2$, and  $0<\theta<(1-\lambda)/5$:
\begin{equation*}
\begin{aligned}
  &\frac{1-\lambda-5\theta}{4}\int_{0}^{t}e^{3\mu(t-s)}\|\hat{U}(s)-U(s)\|^2ds\\
&  \leq|\hat{e}(t)|^2+\big(1-\lambda-5\theta\big)\int_{0}^t e^{3\mu(t-s)}\big(\|e(s)\|^2+\|\hat{e}(s)\|^2\big)ds\\
 &\leq e^{3\mu k_1}|e(0)|^2+\big(e^{3\mu(t-t^1)}-1\big)|e(t^{1})|^2+ \int_{0}^{t} e^{3\mu(t-s)} \Big[ \big(1+\frac{L^2}{4\theta}\big) \|\hat{U}-U\|^2+2\mu|\hat{U} \\
&\ -U|^2\Big]ds  + \frac{1}{\theta}\int_{t^0}^{t^1}e^{3\mu(t^1-s)}  \|R_{b_1}\|_{\star}^2ds +\frac{1}{\theta}\sum_{i=2}^{n-1} \int_{t^{i-1}}^{t^{i}} e^{3\mu(t-s)}  \bigg(\big\|\big(\varphi_1(-k_{i}A)-I\big) R\big\|_{\star}^2  \\
&\ +\|\varphi_1(-k_{i}A)R_b\|_{\star}^2+\Big\|r_{i}\Big[\frac{1}{2}\varphi_1(-k_{i}A) -\varphi_2(-k_{i}A)\Big] \big(B^{i-1} -B^{i-2}\big)\Big\|_{\star}^2\bigg)ds \\
&\   +\frac{1}{\theta} \int_{t^{n-1}}^{t} e^{3\mu(t-s)}  \bigg(\big\|\big(\varphi_1(-k_{n}A)-I\big) R(s)\big\|_{\star}^2 +\|\varphi_1(-k_{n}A)R_b(s)\|_{\star}^2 \\
 &\  + \Big\|r_{n}\Big[\frac{1}{2}\varphi_1(-k_{n}A)-\varphi_2(-k_{n}A)\Big]\big(B^{n-1}-B^{n-2}\big)\Big\|_{\star}^2\bigg)ds,
\end{aligned}
\end{equation*}
where $R(s)$, $R_b(s)$, and $R_{b_1}(s)$  are defined by \eqref{semi linear resi1}, \eqref{semi b}, and \eqref{semi b_1}, respectively.
\end{theorem}
\begin{proof}
Using inequalities \eqref{one-sided Lip} and \eqref{Lip}, for any positive $\theta$, one has
\begin{equation}\label{inequ1}
\begin{aligned}
2\big(B(t,u(t))-B(t,U(t)),\hat{e}(t)\big)&\leq (\lambda+2\theta)\big(\|\hat{e}(t)\|^2+\|e(t)\|^2\big)+3\mu|\hat{e}(t)|^2 \\
&\quad +2\mu |\hat{U}(t)-U(t)|^2+\frac{L^2}{4\theta}\|\hat{U}(t)-U(t)\|^2.
\end{aligned}
\end{equation}
 Inserting \eqref{inequ1} into \eqref{semi error in1}, and using the  Cauchy--Schwarz inequality, we get
\begin{equation*}
\begin{aligned}
\displaystyle &\frac{d}{dt}|\hat{e}(t)|^2-3\mu|\hat{e}(t)|^2+\big(1-\lambda-5\theta \big)\big(\|e(t)\|^2+\|\hat{e}(t)\|^2\big)\\
&\leq \big(1+\frac{L^2}{4\theta}\big)\|\hat{U}(t)-U(t)\|^2+2\mu|\hat{U}(t)-U(t)|^2+\frac{1}{\theta} \bigg(\big\|\big(\varphi_1(-k_{n}A)-I\big)R(t)\big\|_{\star}^2
\\
&\quad +\|\varphi_1(-k_{n}A) R_b(t)\|_{\star}^2
+\Big\|r_{n}\Big[\frac{1}{2}\varphi_1(-k_{n}A)-\varphi_2(-k_{n}A)\Big](B^{n-1}-B^{n-2})\Big\|_{\star}^2\bigg),
\end{aligned}
\end{equation*}
for  $0<\theta<(1-\lambda)/5$. We integrate the above inequality from $t^1$ to $t\in J_{n}$, $n\geq2$,  and obtain
\begin{equation*}
\begin{aligned}
& |\hat{e}(t)|^2+\big(1-\lambda-5\theta\big)\int_{t^1}^t e^{3\mu(t-s)}\big(\|e(s)\|^2+\|\hat{e}(s)\|^2\big)ds\\
\displaystyle
&\leq  e^{3\mu(t-t^1)}|\hat{e}(t^{1})|^2 + \int_{t^1}^{t} e^{3\mu(t-s)}\Big[\big(1+\frac{L^2}{4\theta}\big)\|\hat{U}(s)-U(s)\|^2+2\mu|\hat{U}(s)-U(s)|^2 \Big]ds \\
 	&\quad +\frac{1}{\theta}\sum_{i=2}^{n-1} \int_{t^{i-1}}^{t^{i}} e^{3\mu(t-s)} \bigg( \big\|\big(\varphi_1(-k_{i}A)-I\big)R(s)\big\|_{\star}^2+\|\varphi_1(-k_{i}A)R_b(s)\|_{\star}^2\\
&\quad+\Big\|r_{i}\Big[\frac{1}{2}\varphi_1(-k_{i}A)-\varphi_2(-k_{i}A)\Big]\big(B^{i-1} -B^{i-2}\big)\Big\|_{\star}^2 \bigg)ds  \\
&\quad+\frac{1}{\theta} \int_{t^{n-1}}^{t}e^{3\mu(t-s)}
\bigg( \big\|\big(\varphi_1(-k_{n}A) -I\big)R(s)\big\|_{\star}^2+\|\varphi_1(-k_{n}A)R_b(s)\|_{\star}^2 
\end{aligned}
\end{equation*}
\begin{equation*}
	\begin{aligned}
 +\Big\|r_{n}\Big[\frac{1}{2}\varphi_1(-k_{n}A)-\varphi_2(-k_{n}A)\Big]\big(B^{n-1}-B^{n-2}\big)\Big\|_{\star}^2\bigg)ds.
\end{aligned}
\end{equation*}
Following  the similar argument for deriving  the above inequality,  the upper bound of  the trapezoidal method  over the interval $J_1$ can be  obtained by
\begin{equation}\label{semilinear upper3}
\begin{aligned}
& |\hat{e}(t^1)|^2+\big(1-\lambda-3\theta\big)\int_{t^0}^{t^1} e^{3\mu(t^1-s)}\big(\|e(s)\|^2+\|\hat{e}(s)\|^2\big)ds\\
&\leq e^{3\mu k_1}|\hat{e}(0)|^2+\int_{t^0}^{t^1} e^{3\mu(t^1-s)} \Big[\big(1+\frac{L^2}{4\theta}\big)\|\hat{U}-U\|^2+2\mu|\hat{U}-U|^2+\frac{1}{\theta}\|R_{b_1}(s)\|_{\star}^2 \Big]ds,
\end{aligned}
\end{equation}
where $R_{b_{1}}(s)$ is defined by \eqref{semi b_1}. Consequently, we derive an upper bound for the errors.

Using the relation $\|\hat{U}(s)-U(s)\|\leq \|e(s)\|+\|\hat{e}(s)\|$, it is simple to show that
\begin{equation}\label{semi lower bound}
\frac{1-\lambda-5\theta}{4}\|\hat{U}(s)-U(s)\|^2\leq \big(1-\lambda-5\theta\big)\big(\|e(s)\|^2+\|\hat{e}(s)\|^2\big);
\end{equation}
we integrate inequality \eqref{semi lower bound} from $0$ to $t$, and obtain the desired lower bound. This completes the proof of Theorem \ref{semilinear results}.
\end{proof}

\begin{remark} For a given nonlinearity $B(t,\cdot)$ and a strip $T_u$ along the exact solution, the parameters $\lambda$, $\mu$, and $L$ can be chosen to satisfying conditions \eqref{one-sided Lip} and \eqref{Lip}. The parameter $\theta > 0$ comes from the Young's inequality, with the condition $ \theta <(1-\lambda)/5$  ensuring that 	$(1-\lambda -5\theta)>0$, analogous to the role of $\epsilon_1$ in Theorem \ref{linear problem error estimates}.
	In particular, setting $\mu=0$ and applying Theorem \ref{semilinear results} yields the following estimate:
	\begin{equation}\label{nonlinear estimate2}
		\begin{aligned}
			&\frac{1-\lambda-5\theta}{4}\int_{0}^{T}\|\hat{U}(t)-U(t)\|^2dt\\ &\leq|\hat{e}(T)|^2+\big(1-\lambda-5\theta\big)\int_{0}^T \big(\|e(t)\|^2+\|\hat{e}(t)\|^2\big)dt \\
			&\leq |e(0)|^2+ \big(1+\frac{L^2}{4\theta}\big)\int_{0}^{T}  \|\hat{U}(t)-U(t)\|^2dt+ \frac{1}{\theta}\int_{t^0}^{t^1}  \|R_{b_1}(t)\|_{\star}^2dt \\
			&\quad   +\frac{1}{\theta}\sum_{i=2}^{N} \int_{t^{i-1}}^{t^{i}}   \bigg(\big\|\big(\varphi_1(-k_{i}A)-I\big) R(t)\big\|_{\star}^2 +\|\varphi_1(-k_{i}A)R_b(t)\|_{\star}^2 \\
			&\quad +\Big\|r_{i}\Big[\frac{1}{2}\varphi_1(-k_{i}A) -\varphi_2(-k_{i}A)\Big] \big(B^{i-1} -B^{i-2}\big)\Big\|_{\star}^2\bigg)dt.
		\end{aligned}
	\end{equation}
\end{remark}

\begin{remark}[A priori vs. a posteriori error estimates]
	For parabolic problems \eqref{nonlinear problem}, Theorem \ref{semilinear results} gives the a  posteriori error bounds of  the E-Adams2 method, and the a priori error estimates for $s$-step exponential Adams methods  was derived in  \cite[Theorem 4.3]{Hochbruck2011}.   The essential difference between the a priori and a posteriori error estimates is that  the former depends on the exact solution (which is typically unknown), while the latter depends only on the discretization parameters and the data of the problem and  is computable,  thereby enabling  adaptive computations.
\end{remark}

 \section{An adaptive algorithm}\label{sec5}
Taking the linear parabolic equation as an example, we introduce the adaptive E-Adams2 method based on effective error control. Let $\mathrm{Tol}$ be the tolerance error, given the initial step-size $k_0$, the maximum time step-size $k_{\max}$, the parameters $\delta_1$, $\delta_2$, $\delta_3\in (0,1)$, and the iteration counter $Count$. A posteriori error estimator \eqref{discrete estimates for linear} satisfies the  inequality:
\begin{equation}\label{Tol}
	\begin{aligned}
	\mathrm{Tol} & \geq |e(0)|^2+\int_{0}^{t^N} \|\hat{U}(t)-U(t)\|^2dt +\frac{2\epsilon_1}{3}\int_{t^0}^{t^{1}} \|R_{f_1}(s)\|_{\star}^2ds \\
	&\quad +2\epsilon_1\sum_{n=2}^{N} \int_{t^{n-1}}^{t^{n}}   \bigg( \big\|\big(\varphi_1(-k_{n}A)-I\big)R(t)\big\|_{\star}^2+\|\varphi_1(-k_{n}A)R_f(t)\|_{\star}^2
	\\
	& \quad  + \Big \|r_{n}\big[\frac{1}{2}\varphi_1(-k_{n}A)-\varphi_2(-k_{n}A)\big] \big(f^{n-1}-f^{n-2}\big)\Big \|_{\star}^2 \bigg)dt..
	\end{aligned}
	\end{equation}
For each interval $J_{n}$, $n=1,\ldots,N$, we define the local error estimator  $\mathcal{E}_{\epsilon_1}^{n}$ as follows:

For $n=1$, 
\begin{equation*}
	\mathcal{E}_{\epsilon_1}^{1}=\frac{1}{T}|e(0)|^2+  \frac{k_1^4}{120}\|\bar{\partial}(f^1-AU^1)\|^2+\frac{2\epsilon_1}{k_1} \int_{J_1}\|R_{f_1}(t)\|_{\star}^2dt.\\
\end{equation*}

For $n\geq2$,
\begin{equation*}
	\begin{aligned}
		\mathcal{E}_{\epsilon_1}^{n}=&   \frac{1}{T}|e(0)|^2+ \frac{k_n^4}{120}\|\varphi_1(-k_nA)\bar{\partial} f^{n-1}\|^2+\frac{2\epsilon_1}{k_n} \int_{J_n}   \bigg( \big\|\big(\varphi_1(-k_{n}A)-I\big)R(t)\big\|_{\star}^2\\
		&+\|\varphi_1(-k_{n}A)R_f(t)\|_{\star}^2+
		\Big\|r_{n}\big[\frac{1}{2}\varphi_1(-k_{n}A)-\varphi_2(-k_{n}A)\big] \big(f^{n-1}-f^{n-2}\big)\Big \|_{\star}^2 \bigg)dt.
	\end{aligned}
\end{equation*}
Once  $\mathcal{E}_{\epsilon_1}^{n} \leq \mathrm{Tol}/T$, then inequality \eqref{Tol} holds. At each time step, we compute $\mathcal{E}_{\epsilon_1}^{n}$ and carry out the following step‑size selection:
\begin{itemize}
	\item If $\mathcal{E}_{\epsilon_1}^{n} \le \delta_1 \mathrm{Tol}/T$, the current step‑size $k_n$ is very efficient; we then try to enlarge the step‑size for the next step by setting $k_{n+1} = \min(k_{\max},\, k_n / \delta_2)$.
	\item If $\delta_1\mathrm{Tol}/T < \mathcal{E}_{\epsilon_1}^{n} \le \mathrm{Tol}/T$, the step‑size $k_n$ is acceptable, and we keep $k_{n+1} = k_n$.
	\item Otherwise ($\mathcal{E}_{\epsilon_1}^{n} > \mathrm{Tol}/T$), the current step is rejected; we reduce the step-size $k_n=\delta_3k_n$ and repeat the current step.
\end{itemize}
We provide the pseudocode of the adaptive E-Adams2 method in Algorithm \ref{alg:time step-size}.
\begin{algorithm}
	\caption{Time step-size control}
	\label{alg:time step-size}
	\begin{algorithmic}
		\STATE{ Choose parameters: $\mathrm{Tol}$, $k_0$, $k_{\max}$, $\delta_1$, $\delta_2$, $\delta_3$}
		\STATE{Initialization: $U^0$,  $Count=0$, $t^{0}$}
		\STATE{Set $k_n:=k_{n-1}$}
		\STATE{Compute $U^{n}$ and  $\mathcal{E}_{\epsilon_1}^{n}$}
		\IF{$\mathcal{E}_{\epsilon_1}^{n}\leq \delta_1\mathrm{Tol}/T$}
		\STATE{$t^{n}:=t^{n-1}+k_n$}
		\STATE{$Count=Count+1$}
		\STATE{$k_{n+1}=\min(k_{\max},k_n/{\delta_2}$})
		\ELSIF{$\delta_1\mathrm{Tol}/T <\mathcal{E}_{\epsilon_1}^{n}\leq \mathrm{Tol}/T$}
		\STATE{$t^{n}:=t^{n-1}+k_{n}$}
		\STATE{$Count=Count+1$}
		\STATE{$k_{n+1}=k_{n}$}
		\ELSE
		\STATE{$k_n=\delta_3k_{n}$}
		\ENDIF
	\end{algorithmic}
\end{algorithm}

\begin{remark}
	Algorithm~\ref{alg:time step-size} clearly presents the adaptive E-Adams2 method based on the reliable error bound $\mathcal{E}_{\epsilon_1}^{n}$ for linear problems. The extension to semilinear problems is straightforward: using the a posteriori error estimate \eqref{nonlinear estimate2}, we compute $\mathcal{E}_{\theta}^{n}$, defined analogously to $\mathcal{E}_{\epsilon_1}^{n}$ by replacing $f$ and $\epsilon_1$ by $B$ and $\theta$.
\end{remark}

 \section{Numerical examples}\label{sec6}
 In this section, linear and semilinear numerical examples are studied to illustrate the effectivity of  the  error estimators and and the efficiency of Algorithm \ref{alg:time step-size}. In our computation, the products  of the matrix-valued functions $\varphi_{j}(-k_{n}\mathbf{A})$ and vectors  are calculated by the Pad\'{e} approximation \cite{Berland2007}, and  we use  the Gauss-Legendre quadrature formula with three nodes to evaluate all the integrals from $t^{n-1}$ to $t^{n}$. The parameter  $\epsilon_1$ in Theorem \ref{linear problem error estimates} is problem-dependent and its choice may affects the sharpness of the error bounds.  In the numerical experiments below, we choose  $\epsilon_1=7/4$ for all linear test cases.
    \begin{example}\label{example1}
    We first consider the following  linear problem (see \cite{Chen2004,Huang2013})
    \begin{equation*}\label{linear problem1}\left\{
    \begin{array}{l}
    	u_t(x,t)=u_{xx}(x,t)+f(x,t), \quad x \in [0,1],  \quad t \in[0,1], 
    	\cr\noalign{\vskip1.5truemm}
    	u(0,t)=u(1,t)=0,\qquad \qquad \quad t \in [0,1],
    \end{array}
    \right.\end{equation*}
   where the right term $f(x,t)$ and its initial value are  chosen in such a way that the exact solution is 
   \begin{equation}\label{pro1solution}
   	u(x,t)=0.1\times \sin(\pi x) \times (1-e^{-10000*(t-0.5)^2}).
   	\end{equation}
        \end{example}

    Let us denote the error $e$ at time $T$ by $Err_T:=|e(T)|$,  the  discrete maximum norm and the $L^2(0,T;V)$-norm of the error are defined by
    \begin{equation*}
Err_{\infty}:=\max\limits_{1\leq n\leq N} |e(t^n)|, \qquad  Err_1:=\Big(\int_0^T \|e(s)\|^2ds\Big)^{\frac{1}{2}}.
\end{equation*}
   We denote the square root of $\epsilon_U$ by  $\mathcal{E}_U$, and the a posteriori quantities  $\mathcal{E}_{R_{f}}$, $\mathcal{E}_R$, and $\mathcal{E}_f$ are respectively defined as follows:
   \begin{equation*}
   \begin{aligned}
     \mathcal{E}_{R_{f}}&:=\Big(\int_{t^{0}}^{t^1}\|R_{f_1}(s)\|_{\star}^2+\sum\limits_{n=2}^{N} \int_{t^{n-1}}^{t^{n}} \|\varphi_1(-k_{n}A)R_f(s)\|_{\star}^2ds \Big)^{\frac{1}{2}},\\
     \mathcal{E}_{R}&:=\Big(\sum\limits_{n=2}^{N}\int_{t^{n-1}}^{t^{n}} \big\|\big(\varphi_1(-k_{n}A)-I\big)R(s)\big\|_{\star}^2ds \Big)^{\frac{1}{2}},
     \end{aligned}
   \end{equation*}
   and
   \begin{equation*}
     \mathcal{E}_{f}:=\Big(\sum\limits_{n=2}^{N}k_n \Big\|r_{n}\Big[\frac{1}{2}\varphi_1(-k_{n}A)-\varphi_2(-k_{n}A)\Big]\big(f^{n-1}-f^{n-2}\big)\Big\|_{\star}^2\Big)^{\frac{1}{2}}.
   \end{equation*}
We spatially discretize the operator $A:=-\Delta$ into a symmetric positive definite matrix $\mathbf{A}$, and evaluate all the a posteriori quantities in the fully discrete setting. In this framework, the discrete dual norm is computed as 
\begin{equation*}
\big\|\big(\varphi_1(-k_{n}\mathbf{A})-\mathbf{I}\big)R(s)\big\|_{\star}^2 =h^d\Big[\big(\varphi_1(-k_{n}\mathbf{A})-\mathbf{I}\big)R(s)\Big]^{\intercal}
\mathbf{A}^{-1}\Big[\big(\varphi_1(-k_{n}\mathbf{A})-\mathbf{I}\big)R(s)\Big],
\end{equation*}
where $h$ and $d$ denote the mesh size and dimension, respectively. The other discrete norms are computed similarly.  We then verify the temporal convergence rates of these quantities for the time-discrete scheme. From  \eqref{linear error},  the effectivity indices $ei_{L}$ and $ei_{U}$  of the lower and upper estimators are introduced  by:
\begin{equation*}
  ei_L:=\frac{lower \ estimator}{\sqrt{Err^2_T+{8}Err^2_1/{7}}}, \qquad ei_U:=\frac{upper \  estimator}{\sqrt{Err^2_T+{8}Err^2_1/{7}}};
\end{equation*}
with $lower \ estimator=\sqrt{\mathcal{E}^2_U/8}$ and $upper \ estimator=\sqrt{ \mathcal{E}^2_U+{7}\big(\mathcal{E}^2_{R_{f}}+\mathcal{E}^2_R+\mathcal{E}^2_f\big)/{2}}$.

   The time derivative of solution \eqref{pro1solution} becomes very steep around  $t=1/2$; consequently, the solution varies rapidly over a short time interval. We use the standard  central finite difference  with   the   spatial  mesh size $h_{x}=1/300$  to discretize  the space derivative $u_{xx}$, and  run  the E-Adams2 method \eqref{linear method} with  a uniform temporal step-size $k=1/N$ (i.e., $r_{n}=1$ for $n=2,3,\ldots, N$), $N=128, 256, 512, 1024, 2048$.  The convergence rates of the exact  errors $Err_T$, $Err_{\infty}$, and $Err_{1}$ are shown in Table \ref{table1-1}, and the discrete $L^2$ errors  at all time nodes $t^{n}$ are plotted in Figure \ref{fig5.1} for $N=128$ and $N=2048$.
      The a posteriori quantities $\mathcal{E}_{U}$, $\mathcal{E}_{R_{f}}$, $\mathcal{E}_R$,  and their convergence orders are indicated in  Table \ref{table1-2}. Note   that  the convergence orders of these a posteriori quantities   are  slightly affected by the matrix-valued functions $\varphi_{j}(-k \mathbf{A})$ with relatively larger temporal step-sizes, and   the impact of  the  matrix-valued functions $\varphi_{j}(-k \mathbf{A})$ is decreasing as the temporal step-size decreases.   The convergence  order of the a posteriori quantity $\mathcal{E}_{f}$, and the effectivity indices $ei_L$ and $ei_U$ of the method   \eqref{linear method}   are reported in Table \ref{table1-3}.  We observe that the effectivity index  $ei_L$  of the lower estimator is around $0.07$,
       and  the effectivity index  $ei_U$  of the  upper estimator  is  around $48$. 
       
     Then, we set $k_0=1/50$, $k_{\max}=1/10$, $\delta_1={1}/{4}$, $\delta_2={2}/{3}$, $\delta_3={1}/{2}$, and apply the adaptive E-Adams2 method for Example \ref{example1}. Numerical results are reported in Table \ref{table1-4}, and we show the time step-size trajectory and the $L^2$ error of Algorithm \ref{alg:time step-size} at each step with $\mathrm{Tol}=1/10$ in Figure \ref{fig5.2}.  It can be observed that the time step-size drops sharply around $t=1/2$, coinciding with a peak in the $L^2$ error. Away from $t=1/2$, the step-size can be significantly larger while still satisfying  the condition $upper \ estimator\leq \sqrt{\mathrm{Tol}}$. Numerical results validate a posteriori error analysis established in Section \ref{sec3.2}, and confirm the high efficiency of  Algorithm \ref{alg:time step-size}.
     
\begin{table}[htbp]
	\small
\centering
\caption{The exact errors  $Err_T$, $Err_{\infty}$,  $Err_1$, and their convergence orders for the E-Adams2 method \eqref{linear method} for  Example \ref{example1}}\label{table1-1} \vskip  -2mm
\renewcommand\arraystretch{1.4}
\begin{tabular}{c|c|c|c|c|c|c}
\hline
    N& $Err_{T}$&  order  & $Err_{\infty}$ & order & $Err_1$ &order \\
    \hline
128&6.4482e-07 &  &2.5184e-02 &  &8.3331e-03  &\\
 \hline
256&6.4535e-07 &-	&8.2499e-03 &1.6101  &2.5509e-03  &1.7079\\
 \hline
512&6.4535e-07 &- &2.2034e-03 & 1.9046 &6.7357e-04 &1.9211\\
 \hline
1024&6.4535e-07&-   &5.5915e-04&1.9784  &1.7064e-04 & 1.9809\\
 \hline
2048&6.4535e-07&-  &1.3983e-04&1.9995  & 4.2819e-05  &1.9946\\
 \hline
       \end{tabular}
\end{table}

\begin{table}[htbp]
\small
\centering
\caption{The a posteriori  quantities  $\mathcal{E}_{U}$, $\mathcal{E}_{R_{f}}$, $\mathcal{E}_R$,  and their convergence orders for the E-Adams2 method \eqref{linear method}  for  Example \ref{example1}}\label{table1-2} \vskip -2mm
\renewcommand\arraystretch{1.4}
\begin{tabular}{c|c|c|c|c|c|c}
\hline
N& $\mathcal{E}_{U}$    & order  & $\mathcal{E}_{R_{f}}$  & order   & $\mathcal{E}_R$ & order     \\
    \hline

128  &2.0181e-03 &   &2.1645e-01  &    &7.9083e-03 &   \\
 \hline
256&5.7119e-04 &1.8210  &6.9191e-02  &1.6453    &1.4462e-03   & 2.4511   \\
 \hline
512&1.4763e-04 & 1.9520  &1.8423e-02  &1.9091     &2.7693e-04   &2.3846 \\
 \hline
1024&3.7306e-05&1.9845  &4.6910e-03  &1.9735   &6.1914e-05   & 2.1612 \\
 \hline
2048& 9.3630e-06&1.9944   & 1.1796e-03 &1.9916    & 1.4979e-05  &2.0473   \\
 \hline
       \end{tabular}
\end{table}

\begin{table}[htbp]
\small
\centering
\caption{The a posteriori quantity   $\mathcal{E}_{f}$ and its convergence order,  lower and upper estimators, and the  effectivity indices for the E-Adams2 method \eqref{linear method} for  Example \ref{example1}}\label{table1-3} \vskip -2mm
\renewcommand\arraystretch{1.3}
\begin{tabular}{c|c|c|c|c|c|c}
\hline
N &$\mathcal{E}_f$  &order   &lower estimator   &upper estimator    & $ei_{L}$ & $ei_U$ \\
    \hline
128&1.8421e-03 &   &7.1351e-04 &4.0522e-01  &0.0801  & 45.4874 \\
 \hline
256& 5.2141e-04& 1.8209  & 2.0194e-04  &1.2948e-01  &0.0741 &47.4805\\
 \hline
512& 1.3477e-04&  1.9519  &5.2195e-05 &3.4472e-02 & 0.0725&47.8724\\
 \hline
1024&3.4056e-05& 1.9845 &1.3190e-05 & 8.7772e-03  &0.0723 & 48.1153\\
 \hline
2048&8.5472e-06&  1.9944 &3.3103e-06 &2.2071e-03 &0.0723 & 48.2105 \\
 \hline
       \end{tabular}
\end{table}

\begin{figure}[ht]
	\centering
	\begin{subfigure}[c]{0.48\textwidth}
		\includegraphics[width=\textwidth, height=5.7cm]{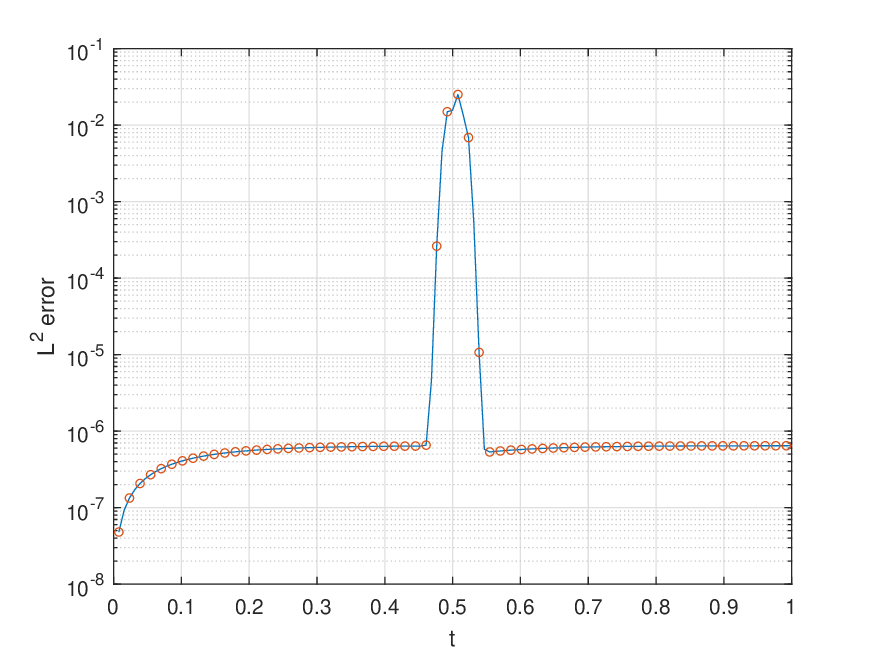}
	\end{subfigure}
	\hfill
	\begin{subfigure}[c]{0.48\textwidth}
		\includegraphics[width=\textwidth, height=5.7cm]{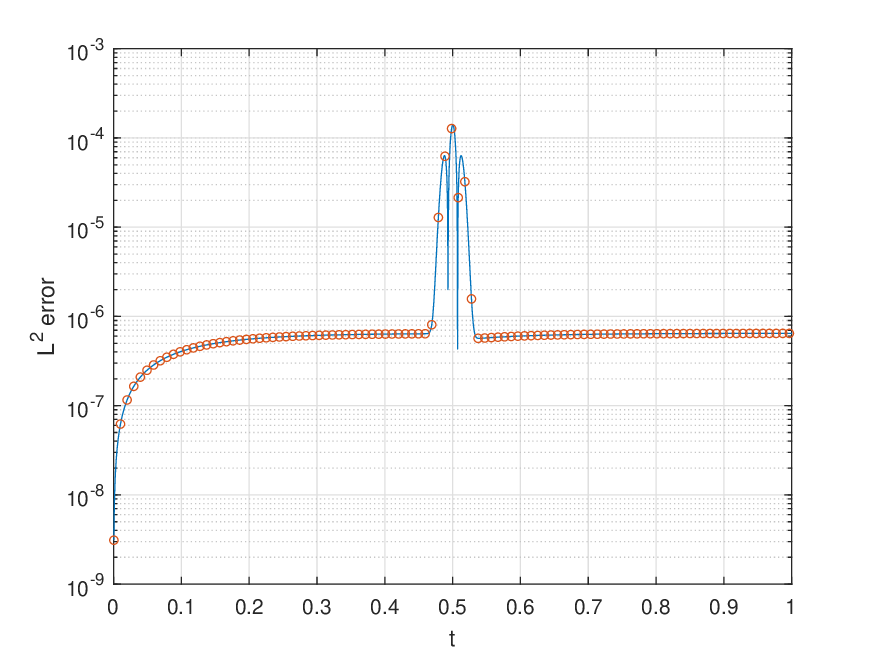}
	\end{subfigure}
	\caption{Example \ref{example1}: The discrete $L^2$ error of the E-Adams2 method for $N=128$ (left) and $N=2048$ (right)}\label{fig5.1}
\end{figure}

\begin{figure}[ht]
	\centering
	\begin{subfigure}[c]{0.49\textwidth}
		\includegraphics[width=\textwidth, height=5.8cm]{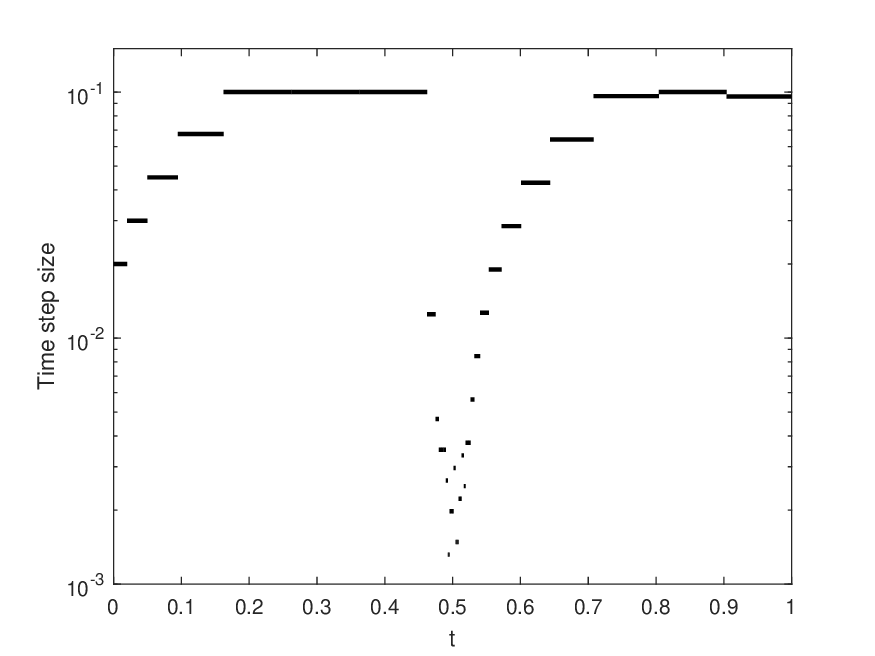}
	\end{subfigure}
	\hfill
	\begin{subfigure}[c]{0.49\textwidth}
		\includegraphics[width=\textwidth, height=5.8cm]{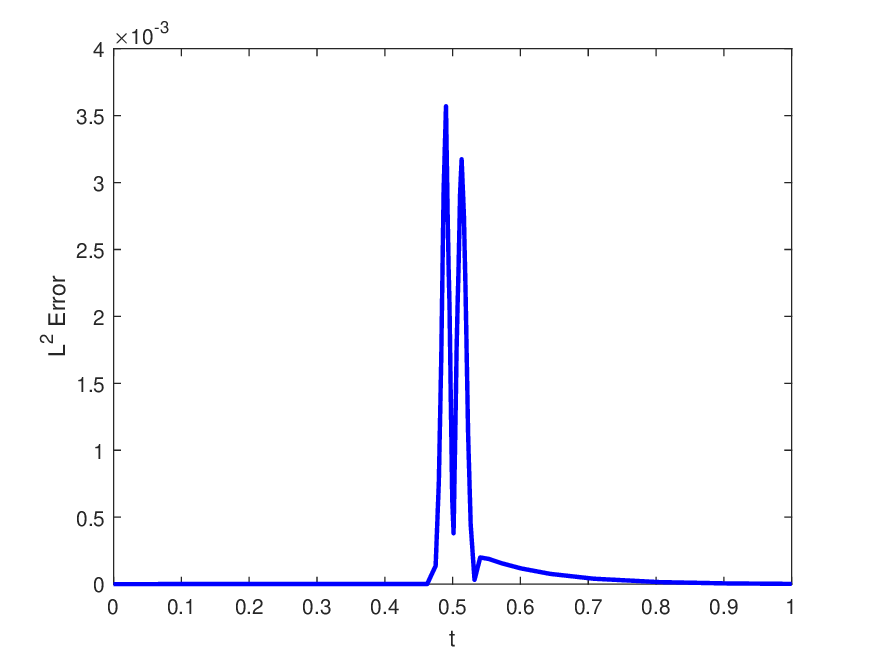} 
	\end{subfigure}
	\caption{Example \ref{example1}: Time step-size trajectory (left) and the $L^2$ error (right) of the adaptive E-Adams2 method with $\mathrm{Tol}=1/10$ }\label{fig5.2}
\end{figure} 

\begin{table}[htbp]
	\small
	\centering
	\caption{Initial parameters and numerical results for the adaptive E-Adams2 method for Example \ref{example1}}\label{table1-4} \vskip  -2mm
	\renewcommand\arraystretch{1.4}
	\begin{tabular}{c|c|c|c|c|c|c}
		\hline
		$\mathrm{Tol}$ & $Err_{T}$&  $Err_{\infty}$  &  $Err_1$ & upper estimator & $ei_U$ &Count\\
		\hline
		1e-01&1.6523e-06 & 3.5699e-03 &1.4995e-03 & 4.1031e-02 &25.5960  &38\\
		\hline
		1e-02&1.3705e-06 &1.2531e-03	&5.7092e-04 &1.4598e-02 &23.9173  &54\\
		\hline
		1e-03&1.9529e-07 &4.4924e-04 &2.0705e-04 & 4.4900e-03 &20.2853&83\\
		\hline
		1e-04&2.5164e-07&1.5060e-04 &7.4698e-05&1.5562e-03  &19.4878 & 130\\
		\hline
		1e-05&4.2756e-07&5.3628e-05 &2.7628e-05&5.4569e-04  & 18.4737  &207\\
		\hline
	   1e-06&5.9183e-07&2.1819e-05  &1.0888e-05&1.9005e-04  & 16.3068 &346\\
	   	\hline
	\end{tabular}
\end{table}

\begin{example}\label{example2}
 	Next, we take equation \eqref{linear problem1} on $[0,1] \times [0,10]$ with the exact solution $u(x,t)=\sin(\pi x)e^{-800\sin(\pi t/2 -1)^2}sin(4\pi t)$.
\end{example}

In this example, we adopt the central finite difference  scheme with a fixed spatial mesh size $h_{x}=1/300$  to approximate $u_{xx}$, and run the E-Adams2 method \eqref{nonlinear method} with  a uniform temporal step-size
$k=10/N$, $N=320, 640, 1280, 2560,5120$.  Table \ref{table2-1} presents the  exact errors
$Err_T$, $Err_{\infty}$,  $Err_1$, and their convergence rates. Figure \ref{fig5.3} shows the discrete $L^2$ errors at all time nodes $t^n$ for $N=320$ and $N=5120$. The a posteriori  quantities  $\mathcal{E}_{U}$, $\mathcal{E}_{R_{f}}$, $\mathcal{E}_R$, $\mathcal{E}_f$,  and their convergence orders are reported  in Tables \ref{table2-2} and \ref{table2-3}. From Table \ref{table2-3}, the effectivity indices  $ei_L$ and $ei_U$ of lower and upper estimators seem to be asymptotically constant (around $0.07$ and $6.0$, respectively). Taking the parameters $k_0=1/10$, $k_{\max}=1/2$, $\delta_1={1}/{5}$, $\delta_2={2}/{3}$, $\delta_3={1}/{2}$, and applying the adaptive algorithm \ref{alg:time step-size} to Example \ref{example2}, we obtain the computational results reported in Table \ref{table2-4}.  Figure \ref{fig5.4} displays the time step-size trajectory and the $L^2$ error at each step with $\mathrm{Tol}=1/10$; the time step-size drops sharply around $t=1$, $5$, and $9$, while the \(L^2\) error reaches its peaks.

For linear problems, the adaptive E-Adams2 method is more efficient, using fewer steps to achieve the same or higher accuracy, than the adaptive integrating factor midpoint method \cite{Hu2026} and the adaptive finite element method \cite{Huang2013}.

\begin{table}[htbp]
	\small
	\centering
	\caption{The exact errors  $Err_T$, $Err_{\infty}$,  $Err_1$, and their convergence orders for the E-Adams2 method \eqref{linear method} for  Example \ref{example2}}\label{table2-1} \vskip  -2mm
	\renewcommand\arraystretch{1.4}
	\begin{tabular}{c|c|c|c|c|c|c}
		\hline
		N& $Err_{T}$&  order  & $Err_{\infty}$ & order & $Err_1$ &order \\
		\hline
	320&8.3076e-10 &  &8.0698e-02 &  &1.7154e-01  &\\
		\hline
		640&8.6562e-10 &-	&1.9899e-02 &2.0198  &4.3820e-02  &1.9689\\
		\hline
		1280&8.7314e-10 &- &4.9129e-03 & 2.0181 &1.0939e-02 & 2.0022\\
		\hline
		2560&8.7489e-10&-   & 1.2223e-03&2.0069  &2.7251e-03 & 2.0051\\
		\hline
		5120&8.7531e-10&-  &3.0620e-04&1.9971  & 6.8122e-04 &2.0001\\
		\hline
	\end{tabular}
\end{table}

\begin{table}[htbp]
	\small
	\centering
	\caption{The a posteriori  quantities  $\mathcal{E}_{U}$, $\mathcal{E}_{R_{f}}$, $\mathcal{E}_R$,  and their convergence orders for the E-Adams2 method \eqref{linear method}  for  Example \ref{example2}}\label{table2-2} \vskip -2mm
	\renewcommand\arraystretch{1.4}
	\begin{tabular}{c|c|c|c|c|c|c}
		\hline
		N& $\mathcal{E}_{U}$    & order  & $\mathcal{E}_{R_{f}}$  & order   & $\mathcal{E}_R$ & order     \\
		\hline
		320  &3.5480e-02 &   &4.8087e-01  &    &8.5454e-02 &   \\
		\hline
		640&9.7374e-03 &1.8654  &1.3697e-01  &1.8117   &1.6800e-02   & 2.3467   \\
		\hline
		1280&2.5404e-03 & 1.9385  &3.6085e-02  &1.9244    &3.8267e-03  &2.1343 \\
		\hline
		2560& 6.4814e-04&1.9707  &9.2291e-03  &1.9672  &9.3606e-04   & 2.0314 \\
		\hline
		5120& 1.6365e-04&1.9857   & 2.3317e-03 &1.9848   &  2.3342e-04  &2.0037  \\
		\hline
	\end{tabular}
\end{table}

\begin{table}[htbp]
	\small
	\centering
	\caption{The a posteriori quantity   $\mathcal{E}_{f}$ and its convergence order,  lower and upper estimators, and effectivity indices for the E-Adams2 method \eqref{linear method} for  Example \ref{example2}}\label{table2-3} \vskip -2mm
	\renewcommand\arraystretch{1.3}
	\begin{tabular}{c|c|c|c|c|c|c}
		\hline
		N &$\mathcal{E}_f$  &order   &lower estimator   &upper estimator    & $ei_{L}$ & $ei_U$ \\
		\hline
		320&3.2337e-02 &   &1.2544e-02 &9.1640e-01  &0.0684  & 4.9972 \\
		\hline
		640& 8.8854e-03& 1.8637  &3.4427e-03 & 2.5889e-01  &0.0735 &5.5265\\
		\hline
		1280& 2.3189e-03&  1.9380  &8.9818e-04 & 6.8074e-02 & 0.0768&5.8213\\
		\hline
		2560&5.9165e-04& 1.9706 &2.2915e-04 & 1.7402e-02 &0.0787 &  5.9735\\
		\hline
		5120&1.4939e-04&  1.9857 &5.7858e-05 & 4.3959e-03 &0.0794& 6.0363 \\
		\hline
	\end{tabular}
\end{table}

\begin{figure}[ht]
	\centering
	\begin{subfigure}[c]{0.48\textwidth}
		\includegraphics[width=\textwidth, height=5.8cm]{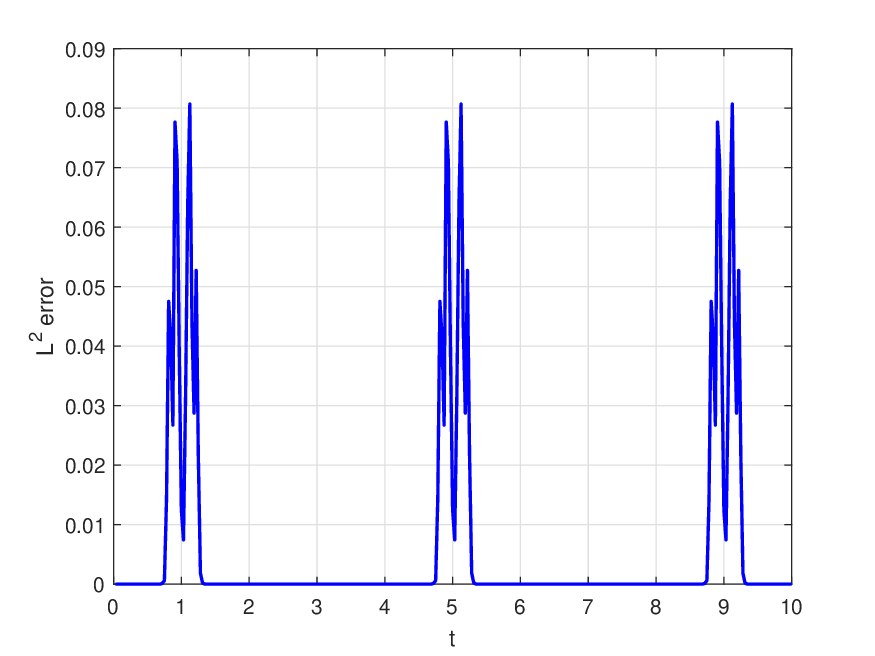}
	\end{subfigure}
	\hfill
	\begin{subfigure}[c]{0.48\textwidth}
		\includegraphics[width=\textwidth, height=5.8cm]{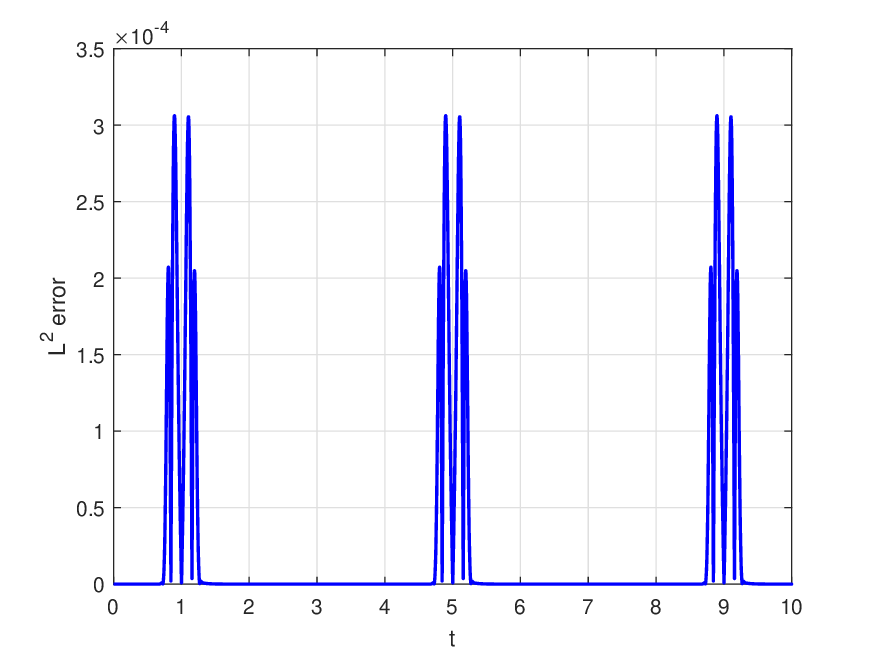}
	\end{subfigure}
	\caption{Example \ref{example2}: The discrete $L^2$ error of the E-Adams2 method for $N=320$ (left) and $N=5120$ (right)}\label{fig5.3}
\end{figure}

\begin{table}[htbp]
	\small
	\centering
	\caption{Initial parameters and numerical results for the adaptive E-Adams2 method for Example \ref{example2}}\label{table2-4} \vskip  -2mm
	\renewcommand\arraystretch{1.4}
	\begin{tabular}{c|c|c|c|c|c|c}
		\hline
		$\mathrm{Tol}$ & $Err_{T}$&  $Err_{\infty}$  &  $Err_1$ & upper estimator & $ei_U$ &Count\\
		\hline
		1e-01&2.5070e-06 &7.9057e-03&1.6988e-02 &7.9970e-02 &4.4034  &212\\
		\hline
		1e-02&4.6659e-07 &2.5679e-03&5.5236e-03 & 2.6560e-02 &4.4980&346\\
		\hline
	  1e-03&1.1958e-07 &8.1046e-04 &1.6126e-03& 8.6699e-03 &5.0289 &580\\
		\hline
		1e-04&5.8134e-08&3.4705e-04 &6.9802e-04&2.8822e-03 &3.8625 & 982\\
		\hline
		1e-05&7.2242e-09&8.2460e-05 &1.6883e-04&9.8746e-04  & 5.4710  &1708\\
		\hline
		1e-06&2.8904e-10&2.7017e-05  &5.4624e-05&3.1356e-04  & 5.3695 &2991\\
		\hline
	\end{tabular}
\end{table}

\begin{figure}[ht]
	\centering
	\begin{subfigure}[c]{0.49\textwidth}
		\includegraphics[width=\textwidth, height=5.8cm]{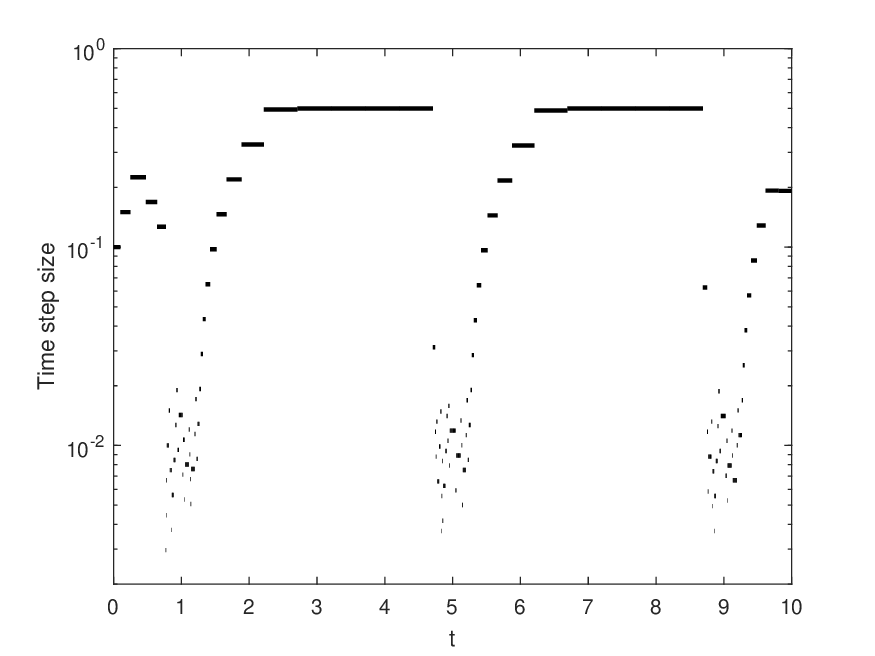}
	\end{subfigure}
	\hfill
	\begin{subfigure}[c]{0.49\textwidth}
		\includegraphics[width=\textwidth, height=5.8cm]{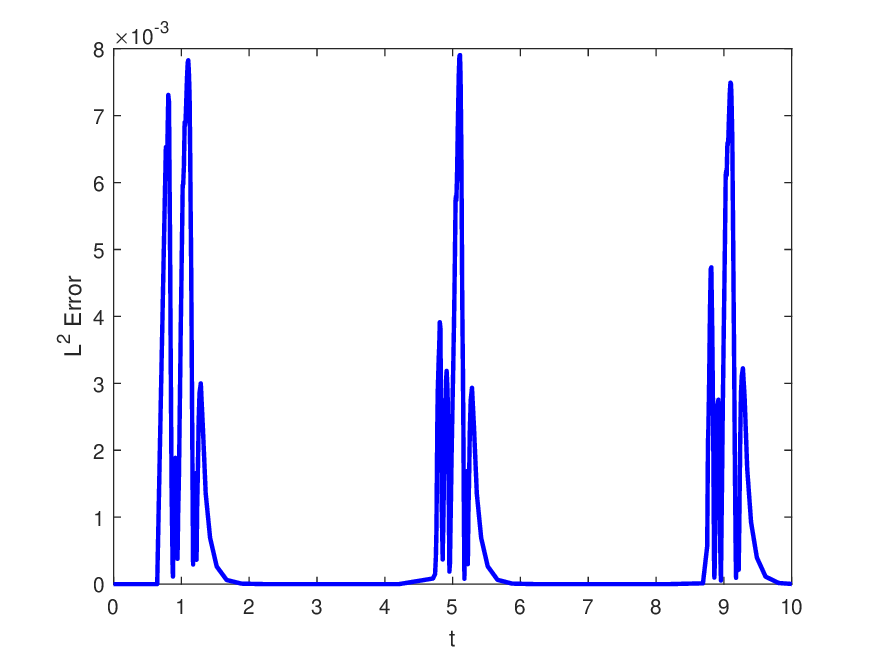} 
	\end{subfigure}
	\caption{Example \ref{example2}: Time step-size trajectory (left) and the $L^2$ error (right) of the adaptive E-Adams2 method with $\mathrm{Tol}=1/10$ }\label{fig5.4}
\end{figure} 

\begin{example}\label{example3}
In the final example, we  consider the following two-dimensional Allen--Cahn equation (see \cite{Ju2021})
in the domain $\Omega=(0,1)^2$:
\begin{equation*}\left\{
\begin{array}{l}
u_t(x,y,t)=\varepsilon^2 \Delta u +b(u),\quad  (x,y)\in \Omega, \quad 0\leq t \leq T,\cr\noalign{\vskip2truemm}
u(x,y,0)=0.1(\sin(3\pi x)\sin(2 \pi y)+\sin(5 \pi x) \sin(5 \pi y)),\quad (x,y)\in \Omega,
\end{array}
\right.\end{equation*}
with the double-well potential   $B(u)=(u^2-1)^2/4$ and $b(u)=-B^{\prime}(u)$.
\end{example}

For the parameter $\varepsilon=0.01$,  the Laplacian operator $\Delta$ is discretized by  the five-point central finite difference  with the spatial mesh size $h_{x}=h_{y}=1/20$. We calculate the E-Adams2 method \eqref{nonlinear method} at  time $T=1$ with the temporal step-size $k=1/N$, $N=30,60,120,240$,  and take the method \eqref{nonlinear method} with  $k=1/3840$ as the reference solution. The convergence rates of the exact errors $Err_T$, $Err_{\infty}$, and  $Err_1$ are shown in Table \ref{table4-1}, and the correct convergence rates are obvious. According to Theorem \ref{semilinear results} and conditions \eqref{one-sided Lip}-\eqref{Lip}, we choose the parameters
 $\lambda=0$, $\mu=1$, $\theta=1/6$, and $L=2$ for the Allen-Cahn equation $u_t=\varepsilon^2 \Delta u+u-u^3$, and denote $\mathrm{E}_{U}$ as
 \begin{equation*}
 	\mathrm{E}_{U}:=\bigg(\int_0^T e^{3\mu(t-s)}|\hat{U}(s)-U(s)|^2ds\bigg)^{\frac{1}{2}}.
 \end{equation*}
 Tables \ref{table4-2} and \ref{table4-3} clearly show the a posteriori quantities  $\mathcal{E}_{U}$, $\mathrm{E}_{U}$, $\mathcal{E}_{R_{b}}$, $\mathcal{E}_R$,  $\mathcal{E}_{b}$, and their second order convergence rates.
Theorem  \ref{semilinear results} indicates $lower \ estimator=\sqrt{\mathcal{E}^2_U/24}$ and $upper \ estimator=\sqrt{(\exp(3)-1)|e(t^1)|^2+7\mathcal{E}^2_U+2 \mathrm{E}^2_{U}+6\big(\mathcal{E}^2_{R_{b}}+\mathcal{E}^2_R+\mathcal{E}^2_b\big)}$, i.e.,
\begin{equation*}
ei_L=\frac{lower \ estimator}{\sqrt{Err^2_T+Err^2_1/3}}, \qquad ei_U=\frac{upper \ estimator}{\sqrt{Err^2_T+Err^2_1/3}}.
\end{equation*}
 It seems that the effectivity indices $ei_{L}$ and $ei_U$ of the method \eqref{nonlinear method}  are  asymptotically constant (around $0.01$ and $38$, respectively).

For semilinear problems,  computational results agree with the a posteriori error estimates established in Theorem \ref{semilinear results}.

\begin{table}[htbp]
\small
\centering
\caption{The exact errors  $Err_T$, $Err_{\infty}$, $Err_1$, and their convergence orders for the E-Adams2 method \eqref{nonlinear method} for  Example \ref{example3}}\label{table4-1} \vskip  -2mm
\renewcommand\arraystretch{1.4}
\begin{tabular}{c|c|c|c|c|c|c}
\hline
    N& $Err_{T}$&  order  & $Err_{\infty}$ & order & $Err_1$ &order \\
    \hline
30&4.5172e-05 &        &4.5172e-05 &         & 6.1260e-06 &  \\
 \hline
60&1.1617e-05 &1.9592  &1.1617e-05 &1.9592   &1.6013e-06 &1.9357  \\
 \hline
120&2.9428e-06 &1.9810  &2.9428e-06 &1.9810  &4.0889e-07  &1.9694 \\
 \hline
240& 7.3858e-07& 1.9944 & 7.3858e-07& 1.9944  &1.0303e-07 &1.9886 \\
 \hline
       \end{tabular}
\end{table}

\begin{table}[htbp]
\footnotesize
\centering
\caption{The a posteriori  quantities  $\mathcal{E}_{U}$,  $\mathrm{E}_{U}$, $\mathcal{E}_{R_{b}}$, $\mathcal{E}_R$,  and their convergence orders for the E-Adams2 method \eqref{nonlinear method}  for  Example \ref{example3}}\label{table4-2} \vskip -2mm
\renewcommand\arraystretch{1.6}
\begin{tabular}{c|c|c|c|c|c|c|c|c}
\hline
N& $\mathcal{E}_{U}$    & order  & $\mathrm{E}_{U}$    & order & $\mathcal{E}_{R_{b}}$  & order   & $\mathcal{E}_R$ & order    \\
\hline
30&  3.4789e-06   &      &2.0280e-05 &        &7.0899e-04 & & 5.3920e-06&            \\
 \hline
60&8.9453e-07 &1.9594  & 5.2125e-06 &1.9600 &1.8275e-04&1.9559&
 1.3703e-06&   1.9764 \\
 \hline
120&2.2679e-07 & 1.9798  &1.3212e-06 &1.9801   &4.6374e-05 &1.9785      &3.4535e-07& 1.9883\\
 \hline
240& 5.7095e-08 &1.9899 &3.3260e-07& 1.9900 &1.1679e-05 &1.9894
 &8.6687e-08&   1.9942 \\
 \hline
       \end{tabular}
\end{table}

\begin{table}[htbp]
\small
\centering
\caption{The a posteriori quantity   $\mathcal{E}_{b}$ and its convergence order,  lower and upper estimators,  and effectivity indices for the E-Adams2 method \eqref{nonlinear method} for  Example \ref{example3}}\label{table4-3} \vskip -2mm
\renewcommand\arraystretch{1.4}
\begin{tabular}{c|c|c|c|c|c|c}
\hline
N &$\mathcal{E}_b$  &order   &lower estimator   &upper estimator    & $ei_{L}$ & $ei_U$ \\
    \hline
30&1.8753e-06 &        &7.1013e-07 &1.7370e-03&0.0157 &38.3354\\
 \hline
60&4.5636e-07& 2.0389  & 1.8259e-07 & 4.4772e-04&0.0157&38.4181 \\
 \hline
120& 1.1241e-07&  2.0215  & 4.6293e-08 &1.1361e-04&0.0157&38.4840\\
 \hline
240&2.7884e-08&  2.0112 &1.1654e-08 &2.8614e-05 &0.0157&38.6164\\
 \hline
       \end{tabular}
\end{table}

\section{Conclusion}\label{sec7} In this paper, we established  optimal order a posteriori error estimates for the explicit E-Adams2 method for parabolic problems. We first presented the continuous approximation $U$  defined  by the linear interpolant of nodal values and its first order residual.  When we applied the usual energy technique to the error equation, we  obtained suboptimal order estimates.
To derive optimal order a posteriori error estimates, we  introduced the E-Adams2 reconstruction $\hat{U}$ of $U$ by making use of   the property of the $\varphi$-functions. Optimal order residual-based a posteriori error estimates were then derived, and we developed an adaptive algorithm based the upper bound. Numerical examples were implemented to confirm the theoretical results.
\bibliographystyle{siamplain}
\bibliography{references}

\end{document}